\documentclass[11pt]{article}
\topmargin by-\headsep  \textheight 236mm
\usepackage{amsmath,amssymb,amsthm,amsfonts}
\usepackage[hypertex,hyperindex]{hyperref}
\usepackage{overpic}
\usepackage{graphicx}
\usepackage{subfigure}
\usepackage{indentfirst}
\usepackage[active]{srcltx}

\numberwithin{equation}{section}
\newtheorem{thm}{Theorem}[section]

\newtheorem{lem}{Lemma}[section]

\newtheorem{alg}{Algorithm}[section]
\theoremstyle{remark}
\newtheorem{rem}{Remark}[section]

\def\to{\rightarrow}
  
\newcommand{\q}{\quad}

\def\m{\mbox}

\def\bb{\begin{equation}} \def\ee{\end{equation}}

\def\beqn{\begin{eqnarray}}  \def\eqn{\end{eqnarray}}

\def\beqnx{\begin{eqnarray*}} \def\eqnx{\end{eqnarray*}}

\begin{document}
\begin{center}
{\Large\bf Convergence of an Adaptive Finite Element Method\\
for Distributed Flux Reconstruction}
\end{center}

\centerline{Yifeng Xu\footnote{Department of Mathematics, Scientific Computing Key Laboratory of
Shanghai Universities and E-Institute for Computational Science of Shanghai Universities, Shanghai
Normal University, Shanghai 200234, China. The research of this
author was in part supported by NSFC (11201307), E-Institute of Shanghai Universities (E03004),
Innovation Program of Shanghai Municipal Education Commission (13YZ059) and Shanghai
Normal University Research Program (SK201202).
({\tt yfxu@shnu.edu.cn})}
\quad and \quad Jun Zou\footnote{Department of
Mathematics, The Chinese University of Hong Kong, Shatin, N.T., Hong
Kong. The work of this author was supported by Hong Kong RGC grants
(Projects 405110 and 404611) and a Direct Grant for Research from Chinese University of
Hong Kong.
({\tt zou@math.cuhk.edu.hk})}}

\begin{abstract}
We shall establish the convergence of an adaptive conforming finite element
method for the reconstruction of the distributed flux in a diffusion
system. The adaptive method is based on a posteriori error
estimators for the distributed flux, state and costate variables. The sequence
of discrete solutions produced by the adaptive algorithm is proved
to converge to the true triplet satisfying the optimality conditions
in the energy norm and the corresponding error estimator converges to zero
asymptotically.
\end{abstract}



\section{Introduction}

The heat flux distributions are of significant practical interest in
thermal and heat transfer problems, e.g., the real-time monitoring
in steel industry \cite{ali} and the visualization by liquid crystal
thermography \cite{dk}. Considering its accurate distribution is
rather difficult to obtain in some inaccessible part of the physical
domain, such as the interior boundary of nuclear reactors and steel
furnaces, engineers attempt to recover the heat flux from some
measured data, which leads naturally to the inverse problem of
reconstructing the distributed heat flux from the measurements on
the accessible part of the boundary or the Cauchy problem for an
elliptic/parabolic equation. Several numerical methods have been
proposed for this classical ill-posed problem, among which the
least-squares formulation
\cite{xiezou} \cite{zk} \cite{zl} has received
intensive investigations and has been implemented by means of the
boundary integral method \cite{zl} and the finite element method
\cite{xiezou}.

However, the story is far from complete from the viewpoint of numerical simulations.
One main challenge is to detect local features of
unknown fluxes accurately and efficiently, particularly in the
presence of non-smooth boundaries and discontinuity or singularity
in fluxes. Compared with the finite element reconstruction over meshes
generated by a uniform refinement, which often requires
formidable computational costs to achieve a high resolution, adaptive finite element
methods (AFEM) are clearly a preferable candidate to remedy the
situation as it is able to retrieve the same result with much fewer
degrees of freedom.

A standard adaptive finite element method consists of successive loops of the form:
\begin{equation}\label{afemloop}
    \mbox{\textsf{SOLVE}}\rightarrow\mbox{\textsf{ESTIMATE}}\rightarrow\mbox{\textsf{MARK}}\rightarrow\mbox{\textsf{REFINE.}}
\end{equation}
That is, one first solves the discrete problem for the finite
element solution on the current mesh, computes the related a
posteriori error estimator, marks elements to be subdivided, and
then refines the current mesh to generate a new finer one.

A major force to drive the process \eqref{afemloop} is the module
\textsf{ESTIMATE}, which relies on some computable quantities
(often called a posteriori estimation), formed by the discrete
solution on the current mesh and given data. Since the pioneer work
\cite{br}, a posteriori error estimations have been extensively
investigated for finite element approximations of direct partial
differential equations and the theory has reached a mature level for
elliptic systems; see the monographs \cite{ao} \cite{ban} \cite{ver} and the
references therein. As far as PDE-based inverse problems are concerned,
there are also some important developments, e.g.,
\cite{bangerth} \cite{beckervexler} \cite{beilinajohnson} \cite{beilina2} \cite{beilina3}
\cite{beilina4} \cite{fyl} \cite{gkv} \cite{lxz}. But a vast amount of literature is
available on PDE-constrained optimal control problems;
see \cite{bkr} \cite{hh} \cite{hhik} \cite{llmt} \cite{liuyan} and references therein,
although inverse problems are quite different in nature due to the
severe instability by data noise.

On the other hand, the study of AFEMs itself is also a research topic
of great interest and has made a substantial progress
in the past decade. Specifically, the convergence and the computational
complexity of an AFEM have been analyzed in depth for
the numerical solution of second order boundary value problems; see
\cite{bdd} \cite{ckns} \cite{cdn} \cite{dor} \cite{mor2} \cite{mor3} \cite{nsv} \cite{siebert} \cite{stev}. But there are still no
developments available for inverse problems.
To our knowledge, the only related work is the one in \cite{ghik} and
it studied the asymptotic error reduction property
of an adaptive finite element approximation
for the distributed control problems with control constraints,
where the adaptive algorithm requires one extra step for
some oscillation terms in the module \textsf{MARK} and the interior
node property in the module \textsf{REFINE}.

In this work, we shall fill in the gap
and establish a first convergence result for an adaptive finite element method
for inverse problems, namely, we shall demonstrate that
both the finite element error (in some appropriate norm) and the estimator
converge to zero when the AFEM is applied to reconstruct the distributed flux on some
inaccessible part of the boundary from partial measurements on an
accessible boundary part. Compared with \cite{ghik} for an optimal control problem,
the algorithm studied here is of the same framework as the standard one for (direct) elliptic
problems (e.g. \cite{ckns} \cite{nsv}), particularly no more marking for oscillation terms as well as no interior node property
is enforced in the module \textsf{MARK} and the module \textsf{REFINE},
therefore it is advantageous to
practical computations. Our basic arguments follow some principles
in \cite{siebert}  \cite{mor3} for a class of linear direct boundary value
problems. In this sense, the current work may be viewed as an extension of \cite{mor3} \cite{siebert} for
the AFEM to inverse problems, but due to the nature of the inverse problem there are some essential technical
differences as mentioned below.
\begin{itemize}
\item
The direct problems of some linear partial differential equations were considered in \cite{mor3} \cite{siebert},
while a nonlinear optimization problem for solving an inverse problem
with the temperature field (state) and the flux (control) coupled in a diffusion equation is the focus of this work, which leads to
a saddle-point system.

\item
In \cite{mor3} \cite{siebert} for linear direct problems, a key observation
is the strong convergence of a sequence of discrete solutions generated by the adaptive process \eqref{afemloop} to some limit,
which is a direct consequence of the standard finite element convergence theory such as the Cea's lemma \cite{ciarlet}.
In contrast, achieving such a result for the inverse problem is highly nontrivial. We shall
view the approximate fluxes generated by \eqref{afemloop} as the minimizers to a discrete optimal system,
and employ some techniques from the nonlinear optimizations to establish the strong convergence of
the adaptive sequence to a minimizer of some limiting optimal system.

\item
The convergence was established in \cite{siebert} by first demonstrating the weak vanishing limit
of a sequence of residuals associated with the adaptive solutions, then proving the strong limit
of the sequence of adaptive solutions is the exact solution.
But this approach does not apply to our current problem
as the exact state and the limiting state depend on the exact flux and the limiting flux respectively.
As a remedy, we shall introduce an auxiliary state depending on the limiting flux to help us realize
the desired convergence.
\end{itemize}

Our convergence theory are basically established in three steps. In the first step,
we shall show the sequence of discrete triplets (the approximate state, costate
and flux) produced by the adaptive algorithm converges strongly to some limiting triplet.
Unlike for the direct problem of differential equations, we need to deal with a nonlinear optimization system
with PDE constraints; see section 4. In the second step, we will prove the limiting triplet is the exact one. 
To do so, we have to consider and study
the limiting behaviors of the residuals associated with the approximate state and
costate and introduce an auxiliary problem to resolve a technical difficulty; see section 5.2.
Finally in the last step, we will demonstrate that the error estimator has a vanishing limit.
This will be the consequence of the previous steps and the efficiency of the error
estimator; see the proof of Theorem \ref{thm_convergence_estimator}.

The rest of this paper is organized as follows. In section 2, we
give a description of the flux reconstruction problem and its finite
element method. A standard adaptive algorithm based on an existing
residual-type a posteriori error estimator is stated in section 3.
In section 4, we prove the sequence of discrete triplets converges
to some limiting triplet. The main results are presented in
section 5 and finally the paper is ended with some concluding remarks in section 6.

Throughout the paper we adopt the standard notation for the Lebesgue space $L^{\infty}(G)$ and
$L^{2}$-based Sobolev spaces $H^{m}(G)$ on an open bounded domain $G\subset\mathbb{R}^{d}$. Related norms and semi-norms of $H^{m}(G)$
as well as the norm of $L^{\infty}(G)$ are denoted by $\|\cdot\|_{m,G}$, $|\cdot|_{m,G}$ and $\|\cdot\|_{\infty,G}$ respectively.
We use
$(\cdot,\cdot)_{G}$ to denote the $L^{2}$ scalar product $G$, and the subscript is omitted when
no confusion is caused. Moreover, we shall use $C$, with or without subscript, for
a generic constant independent of the mesh size and it may take a different value
at each occurrence.

\section{Mathematical formulations}
Let $\Omega\subset\mathbb{R}^{d}$ ($d=2,3$) be an open and bounded
polyhedral domain. The boundary $\Gamma$ of $\Omega$ is made up of
two disjoint parts $\Gamma_{a}$ and $\Gamma_{i}$ such that
$\Gamma=\Gamma_{a}\cup\Gamma_{i}$, where
$\Gamma_{a}$ and $\Gamma_{i}$ are the accessible
and inaccessible parts respectively.
The governing diffusion system of our interest is of the form
\begin{equation}\label{diffeq_state}
    -\boldsymbol{\nabla}\cdot(\alpha\boldsymbol{\nabla}u)=f\quad\mbox{in
}\Omega,
\end{equation}
\begin{equation}\label{bc_state}
    \alpha\dfrac{\partial u}{\partial n}+\gamma u=\gamma u_{a}\quad\mbox{on }\Gamma_{a}\,;\quad
    \alpha\dfrac{\partial u}{\partial n}=-q\quad\mbox{on }\Gamma_{i},
\end{equation}
where $\boldsymbol{n}$ is the unit outward normal on $\Gamma$ and
the given data include the source $f\in L^{2}(\Omega)$, the ambient
temperature $u_a\in L^{2}(\Gamma_{a})$, the heat transfer
coefficient $\gamma>0$ and the diffusivity coefficient $\alpha>0$.
For simplicity $\gamma$ and $\alpha$ are both assumed
to be constants, but it is straightforward to extend all our analyses and results
to the case when both are variable functions.
The inverse problem is to recover the distributed flux
$q$, when the partial measurement data $z$ of temperature $u$ is available
on $\Gamma_a$.
We note this problem is highly ill-posed since the Cauchy data $z$
imposed on $\Gamma_{a}$ is inevitably contaminated with
observation errors in practice \cite{xiezou}. To overcome
this difficulty, we often formulate it as a constrained minimization
problem with the Tikhonov regularization:
\begin{equation}\label{constrained_min_cont}
    \min_{q\in
L^{2}(\Gamma_{i})}\mathcal{J}(q)=\frac{1}{2}\|u(q)-z\|^{2}_{0,\Gamma_{a}}+\frac{\beta}{2}\|q\|^{2}_{0,\Gamma_{i}},
\end{equation}
where $u:=u(q)\in H^{1}(\Omega)$ satisfies the variational
formulation of \eqref{diffeq_state}-\eqref{bc_state}:
\begin{equation}\label{vp_state_constraint}
    a(u,\phi)=(f,\phi)+(\gamma u_{a},\phi)_{\Gamma_{a}}-(q,\phi)_{\Gamma_{i}}\quad\forall~\phi\in
H^{1}(\Omega)
\end{equation}
and the constant $\beta>0$ is the regularization parameter. Here
$a(\cdot,\cdot)=:(\alpha\boldsymbol{\nabla}\cdot,\boldsymbol{\nabla}\cdot)+(\gamma\cdot,\cdot)_{\Gamma_{a}}$
is a weighted inner product over $H^{1}(\Omega)$ and its induced
norm $\|\cdot\|_{a}$ is equivalent to the usual $H^{1}$-norm due to
the Poincar\'{e} inequality. There exists a unique minimizer to
the system \eqref{constrained_min_cont}-\eqref{vp_state_constraint}
\cite{xiezou}. Moreover, with a costate $p^{\ast}\in H^{1}(\Omega)$
involved, the minimizer $(q^{\ast},u^{\ast}(q^{\ast}))$ is
characterized by the following optimality conditions \cite{lxz}:
\begin{align}
    &a(u^{\ast},\phi)=(f,\phi)+(\gamma u_{a},\phi)_{\Gamma_{a}}-(q^{\ast},\phi)_{\Gamma_{i}}\quad\forall~\phi\in
H^{1}(\Omega)\label{vp_state}\\
    &a(p^{\ast},v)=(u^{\ast}-z,v)_{\Gamma_{a}}\quad\forall~v\in H^{1}(\Omega),\label{vp_costate}\\
    & (\beta q^{\ast}-p^{\ast},w)_{\Gamma_{i}}=0\quad\forall~w\in L^{2}(\Gamma_{i}).\label{gateaux_cont}
\end{align}

Next we introduce a finite element method to approximate the continuous
problem \eqref{constrained_min_cont}-\eqref{vp_state_constraint}.
Let $\mathcal{T}_{h}$ be a shape-regular conforming triangulation of
$\bar{\Omega}$ into a set of disjoint closed simplices,
with the diameter $h_{T}:=|T|^{1/d}$ for each $T\in\mathcal{T}_{h}$.
Let $V_{h}$ be the usual $H^{1}$-conforming linear element space over $\mathcal{T}_{h}$,
and $V_{h,\Gamma_{i}}: = V_{h}|_{\Gamma_{i}}$ be the feasible discrete space for $q$.
Then the minimization \eqref{constrained_min_cont}-\eqref{vp_state_constraint}
is approximated by
\begin{equation}\label{constrained_min_disc}
    \min_{q_{h}\in V_{h,\Gamma_{i}}}\mathcal{J}(q_{h})=\frac{1}{2}\|u_{h}(q_{h})-z\|^{2}_{0,\Gamma_{a}}+\frac{\beta}{2}\|q_{h}\|^{2}_{0,\Gamma_{i}},
\end{equation}
where $u_{h}:=u_{h}(q_{h})\in V_{h}$ solves the discrete problem
\begin{equation}\label{vp_state_constraint_disc}
    a(u_{h},\phi_{h})=(f,\phi_{h})+(\gamma u_{a},\phi_{h})_{\Gamma_{a}}-(q_{h},\phi_{h})_{\Gamma_{i}}\quad\forall~\phi_{h}\in
V_{h}.
\end{equation}
As in the continuous case, there exists a unique minimizer to
\eqref{constrained_min_disc}-\eqref{vp_state_constraint_disc},  and
the minimizer $q^{\ast}_{h}\in V_{h,\Gamma_{i}}$, the discrete state and costate $u^{\ast}_{h}\in V_h$
and $p^{\ast}_{h}\in V_h$ satisfy the optimality conditions:
\begin{align}
    &a(u_{h}^{\ast},\phi_{h})=(f,\phi_{h})+(\gamma u_{a},\phi_{h})_{\Gamma_{a}}-(q^{\ast}_{h},\phi_{h})_{\Gamma_{i}}\quad\forall~\phi_{h}\in
V_{h}\label{vp_state_disc}\\
    &a(p^{\ast}_{h},v_{h})=(u^{\ast}_{h}-z,v_{h})_{\Gamma_{a}}\quad\forall~v_{h}\in V_{h},\label{vp_costate_disc}\\
    & (\beta q^{\ast}_{h}-p^{\ast}_{h},w_{h})_{\Gamma_{i}}=0\quad\forall~w_{h}\in V_{h,\Gamma_{i}}.\label{gateaux_disc}
\end{align}

\section{A posteriori error estimation and an adaptive algorithm}

In this section we review a residual-type a posteriori error
estimate and a related adaptive algorithm developed in \cite{lxz}.
For this purpose, some more notation and definitions are needed.

The collection of all faces (resp. all interior faces) in $\mathcal{T}_{h}$
is denoted by $\mathcal{F}_{h}$ (resp. $\mathcal{F}_{h}(\Omega)$) and its restriction on $\Gamma_{a}$ and $\Gamma_{i}$ by $\mathcal{F}_{h}(\Gamma_{a})$ and $\mathcal{F}_{h}(\Gamma_{i})$ respectively. The scalar $h_{F}:=|F|^{1/(d-1)}$ stands for
the diameter of $F\in\mathcal{F}_{h}$, which is
associated with a fixed normal unit vector $\boldsymbol{n}_{F}$ in the interior of $\Omega$ and $\boldsymbol{n}_{F}=\boldsymbol{n}$ on the boundary $\Gamma$. We use $D_{T}$ (resp.\,$D_{F}$)
for the union of all elements in $\mathcal{T}_{h}$ with
non-empty intersection with element $T\in\mathcal{T}_{h}$ (resp.\,$F\in\mathcal{F}_{h}$).
Furthermore, for any $T\in\mathcal{T}_{h}$ we denote by $\omega_{T}$
the union of elements in $\mathcal{T}_{h}$ sharing a common
face with $T$, while for any $F\in\mathcal{F}_{h}(\Omega)$ (resp.
$F\in\mathcal{F}_{h}(\Gamma_{a})\cup\mathcal{F}_{h}(\Gamma_{i})$)
we denote by $\omega_{F}$ the union of two elements in
$\mathcal{T}_{h}$ sharing the common face $F$ (resp. the element
with $F$ as an common edge).

For any $(\phi_{h},v_{h},w_{h})\in V_{h}\times V_{h}\times V_{h,\Gamma_{i}}$,
we define two element residuals for each $T\in\mathcal{T}_{h}$ by
\[
    R_{T,1}(\phi_{h})=f+\boldsymbol{\nabla}\cdot(\alpha\boldsymbol{\nabla}\phi_{h})\,
    \quad \mbox{and}
    \quad R_{T,2}(v_{h})=-\boldsymbol{\nabla}\cdot(\alpha\boldsymbol{\nabla}v_{h})\;,
\]
and two face residuals  for each face $F\in\mathcal{F}_{h}$ by
\[
    J_{F,1}(\phi_{h},w_{h})=\left\{\begin{array}{lll}
                        [\alpha\boldsymbol{\nabla} \phi_{h}\cdot\boldsymbol{n}_{F}]\quad&
                        \m{for} ~~F\in\mathcal{F}_{h}(\Omega),\\
                        \gamma u_{a}-\gamma \phi_{h}-\alpha\boldsymbol{\nabla}\phi_{h}\cdot\boldsymbol{n}_{F}\quad&
                        \m{for} ~~F\in\mathcal{F}_{h}(\Gamma_{a}),\\
                        -w_{h}-\alpha\boldsymbol{\nabla} \phi_{h}\cdot\boldsymbol{n}_{F}\quad&
                         \m{for} ~~ F\in\mathcal{F}_{h}(\Gamma_{i})
                \end{array}\right.
\]
and
\[
    J_{F,2}(v_{h},\phi_{h})=\left\{\begin{array}{lll}
                        [\alpha\boldsymbol{\nabla} v_{h}\cdot\boldsymbol{n}_{F}]\quad&
                        \m{for} ~~F\in\mathcal{F}_{h}(\Omega),\\
                        \phi_{h}-z-\gamma v_{h}-\alpha\boldsymbol{\nabla} v_{h}\cdot\boldsymbol{n}_{F}\quad&
                        \m{for} ~~F\in\mathcal{F}_{h}(\Gamma_{a}),\\
                        -\alpha\boldsymbol{\nabla} v_{h}\cdot\boldsymbol{n}_{F}\quad&
                        \m{for} ~~F\in\mathcal{F}_{h}(\Gamma_{i}),
                \end{array}\right.
\]
where $[\alpha\boldsymbol{\nabla}\phi_{h}\cdot\boldsymbol{n}_{F}]$
and $[\alpha\boldsymbol{\nabla}v_{h}\cdot\boldsymbol{n}_{F}]$ are
the jumps across $F\in\mathcal{F}_{h}$. Then for any
$\mathcal{M}_{h}\subseteq\mathcal{T}_{h}$, we introduce
the error estimator
\[
    \begin{split}
        \eta_{h}^{2}(\phi_{h},v_{h},w_{h},f,u_{a},z,\mathcal{M}_{h})&:=\sum_{T\in\mathcal{M}_{h}}\eta_{T,h}^{2}(\phi_{h},v_{h},w_{h},f,u_{a},z)\\
                                                                    &:=\sum_{T\in\mathcal{M}_{h}}(\eta^{2}_{T,h,1}(\phi_{h},w_{h},f,u_{a})+\eta^{2}_{T,h,2}(v_{h},\phi_{h},z))
    \end{split}
\]
with
\[
    \eta^{2}_{T,h,1}(\phi_{h},w_{h},f,u_{a}):=h_{T}^{2}\|R_{T,1}(\phi_{h})\|_{0,T}^{2}+\sum_{F\subset\partial
T}h_{F}\|J_{F,1}(\phi_{h},w_{h})\|^{2}_{0,F}
\]
and
\[
    \eta^{2}_{T,h,2}(v_{h},\phi_{h},z):=h_{T}^{2}\|R_{T,2}(v_{h})\|_{0,T}^{2}+\sum_{F\subset\partial
T}h_{F}\|J_{F,2}(v_{h},\phi_{h})\|^{2}_{0,F}\,,
\]
and the following oscillation errors that involve the given data and
the related elementwise projections:
\[
    \mathrm{osc}_{h}^{2}(f,\mathcal{M}_{h}):=\sum_{T\in\mathcal{M}_{h}}h_{T}^{2}\|f-\bar{f}_{T}\|_{0,T}^{2},\quad
\]
\[
    \mathrm{osc}_{h}^{2}(\phi_{h},w_{h},\mathcal{S}_{h}):=\sum_{F\in\mathcal{S}_{h}}h_{F}\|J_{F1}(\phi_{h},w_{h})-\bar{J}_{F1}(\phi_{h},w_{h})\|^{2}_{0,F},
\]
\[
    \mathrm{osc}_{h}^{2}(v_{h},\phi_{h},\mathcal{S}_{h}):=\sum_{F\in\mathcal{S}_{h}}h_{F}\|J_{F2}(v_{h},\phi_{h})-\bar{J}_{F2}(v_{h},\phi_{h})\|^{2}_{0,F}
\]
for some $\mathcal{M}_{h}\subseteq\mathcal{T}_{h}$ and
$\mathcal{S}_{h}\subseteq\mathcal{F}_{h}$, where $\bar{f}_{T}$ (resp.
$\bar{J}_{F1}(\phi_{h},w_{h})$ and $\bar{J}_{F2}(v_{h},\phi_{h})$)
is the integral average of $f$ (resp. $J_{F1}(\phi_{h},w_{h})$ and
$J_{F2}(v_{h},\phi_{h} )$) over $T$ (resp. $F$). When
$\mathcal{M}_{h}=\mathcal{T}_{h}$ or
$\mathcal{S}_{h}=\mathcal{F}_{h}$, $\mathcal{M}_{h}$ or $\mathcal{S}_{h}$
will be dropped in the parameter list of the error estimator or the oscillation errors above.

With the above preparations, we are now ready to present the upper and lower bounds for
the errors of the finite element solutions in terms of a residual-type estimator
\cite{lxz}.

\begin{thm}\label{thm_reliability} Let $(u^{\ast},p^{\ast},q^{\ast})$ and
$(u^{\ast}_{h},p^{\ast}_{h},q^{\ast}_{h})$ be the solutions of
\eqref{vp_state}-\eqref{gateaux_cont} and
\eqref{vp_state_disc}-\eqref{gateaux_disc} respectively, then there
exists a constant $C$ depending on the shape-regularity
of $\mathcal{T}_{h}$ and the coefficients $\alpha$ and $\gamma$, such that
\begin{equation}\label{global_upperbound}
    \|u^{\ast}-u^{\ast}_{h}\|_{1}^{2}+\|p^{\ast}-p^{\ast}_{h}\|_{1}^{2}+\|q^{\ast}-q^{\ast}_{h}\|_{0}^{2}\leq
    C\beta^{-2}\eta_{h}^{2}(u^{\ast}_{h},p^{\ast}_{h},q^{\ast}_{h},f,u_{a},z).
\end{equation}
\end{thm}

\begin{thm}\label{thm_efficiency} There exists a constant $C$ depending
on the shape-regularity of $\mathcal{T}_{h}$ and the coefficients $\alpha$ and $\gamma$, such that for any
$T\in\mathcal{T}_{h}$,
\begin{equation}\label{local_efficiency}
    \begin{split}
        \eta_{T,h}^{2}(u^{\ast}_{h},p^{\ast}_{h},q^{\ast}_{h},f,u_{a},z)\leq C(&\|u^{\ast}-u^{\ast}_{h}\|^{2}_{0,\omega_{T}}+\|p^{\ast}-p^{\ast}_{h}\|^{2}_{0,\omega_{T}}+\|q^{\ast}-q^{\ast}_{h}\|^{2}_{0,\partial
                            T\cap\Gamma_{i}}\\
                           &\mathrm{osc}^{2}_{h}(f,\omega_{T})+\mathrm{osc}^{2}_{h}(u^{\ast}_{h},q^{\ast}_{h},\partial T)+\mathrm{osc}^{2}_{h}(p^{\ast}_{h},u^{\ast}_{h},\partial
T)).
    \end{split}
\end{equation}
\end{thm}

Based on the error estimators provided in Theorems\,\ref{thm_reliability} and \ref{thm_efficiency} above,
the following adaptive algorithm was proposed for the flux reconstruction
in \cite{lxz}. In what follows the dependence on the triangulations is
indicated by the number $k$ of the mesh refinements.

\begin{alg}
\label{alg_afem_flux-recon} Given a parameter $\theta\in[0,1]$
and a conforming initial mesh $\mathcal{T}_{0}$. Set $k:=0$.
\begin{enumerate}
    \item\emph{\textsf{(SOLVE)}} Solve the discrete problems
    \eqref{vp_state_disc}-\eqref{gateaux_disc} on $\mathcal{T}_{k}$ for
    $(u^{\ast}_{k},p^{\ast}_{k},q^{\ast}_{k})\in V_{k}\times V_{k}\times V_{k,\Gamma_{i}}$.
    \item\emph{\textsf{(ESTIMATE)}} Compute the error estimator
    $\eta_{k}(u^{\ast}_{k},p^{\ast}_{k},q^{\ast}_{k},f,u_{a},z)$.
    \item\emph{\textsf{(MARK)}} Mark a subset
    $\mathcal{M}_{k}\subset\mathcal{T}_{k}$ such that
    \begin{equation}\label{maximum_marking}
        \forall~T\in\mathcal{M}_{k}\quad\eta_{T,k}(u^{\ast}_{k},p^{\ast}_{k},q^{\ast}_{k},f,u_{a},z)\geq
        \theta\max_{T\in\mathcal{T}_{k}}\eta_{T,k}(u^{\ast}_{k},p^{\ast}_{k},q^{\ast}_{k},f,u_{a},z).
    \end{equation}
    \item\emph{\textsf{(REFINE)}} Refine each triangle
    $T\in\mathcal{M}_{k}$ by the newest vertex bisection to get
    $\mathcal{T}_{k+1}$.
    \item Set $k:=k+1$ and go to Step 1.
\end{enumerate}
\end{alg}

A stopping criterion is normally included after step 2 to terminate the
iteration, which is omitted here for the notational convenience.
The maximum strategy \cite{br}, one of the most
common marking criteria, is used in the module \textsf{MARK} and we will
discuss more about other strategies in section 5.3. In addition,
the newest vertex bisection in the module \textsf{REFINE} guarantees the
uniform shape-regularity of $\{\mathcal{T}_{k}\}$ \cite{koss}
\cite{maubach} \cite{mitc} \cite{stev1} \cite{traxler} \cite{ver}.
In other words, all constants only depend on the initial mesh and
the given data but not on any particular mesh in the sequel. We
point out a practically important feature in our algorithm:
the additional marking for oscillation errors
and the interior node property for the refinement are both not required,
which are needed in the adaptive algorithm for an optimal control problem in \cite{ghik}. 
Finally, as the solution $(u^{\ast}_{k},q^{\ast}_{k})\in V_{k}\times V_{k,\Gamma_{i}}$ is
also the minimizer to problem
\eqref{constrained_min_disc}-\eqref{vp_state_constraint_disc} with
$h=k$, we shall view both of them as the same unless specified otherwise.

The adaptive Algorithm \ref{alg_afem_flux-recon} was implemented and analysed in \cite{lxz}.
Several nontrivial numerical examples
were tested there, with different types of singular fluxes, including fluxes with large jumps,
shape-spike fluxes and dipole-like fluxes. From the numerical experiments we have observed
that Algorithm \ref{alg_afem_flux-recon} is able to locate
the singularities of fluxes accurately, with the desired local mesh refinements around singularities.
Moreover, all the examples in \cite{lxz} have shown that
Algorithm \ref{alg_afem_flux-recon} ensures the convergence of the flux errors in $L^{2}$-norm,
even with essentially fewer degrees of freedom than the uniform refinement.
The aim of this work is to provide a rigorous mathematical justification of
the convergence of the adaptive finite element Algorithm \ref{alg_afem_flux-recon}.

\section{A limiting triplet}\label{sec:limit}

In this section, we demonstrate the convergence of the sequence
$\{(u^{\ast}_{k},p^{\ast}_{k},q^{\ast}_{k})\}$ generated by
Algorithm \ref{alg_afem_flux-recon}. To this
end, with $\{V_{k}\}$ and $\{V_{k,\Gamma_{i}}\}$ induced by Algorithm
\ref{alg_afem_flux-recon}, we define two limiting spaces:
$$
V_{\infty}:=\overline{\bigcup_{k\geq 0}V_{k}} ~~(\m{in} ~H^{1}\m{-norm}) \quad and
\quad
Q_{\infty}:=\overline{\bigcup_{k\geq 0}V_{k,\Gamma_{i}}} ~~(\m{in} ~L^{2}\m{-norm)}\,.
$$
We remark that $V_{\infty}$ and $Q_{\infty}$ are  a closed
subspace of $H^{1}(\Omega)$ and $L^{2}(\Gamma_{i})$ respectively. Then
we introduce a constrained minimization problem over
$Q_{\infty}$:
\begin{equation}\label{constrained_min_med}
    \min_{q\in Q_{\infty}}\mathcal{J}(q)=\frac{1}{2}\|u_{\infty}(q)-z\|^{2}_{0,\Gamma_{a}}+\frac{\beta}{2}\|q\|^{2}_{0,\Gamma_{i}},
\end{equation}
where $u_{\infty}:=u_{\infty}(q)\in V_{\infty}$ satisfies the
variational problem:
\begin{equation}\label{vp_state_constraint_med}
    a(u_{\infty},\phi)=(f,\phi)+(\gamma u_{a},\phi)_{\Gamma_{a}}-(q,\phi)_{\Gamma_{i}}\quad\forall~\phi\in
V_{\infty}.
\end{equation}

Following the arguments of \cite{xiezou} for
the system \eqref{constrained_min_cont}-\eqref{vp_state_constraint},
we can show that there exists a unique
minimizer to the optimization problem
\eqref{constrained_min_med}-\eqref{vp_state_constraint_med}.

Next we present the first result of this section,
namely the sequence
${q_{k}^{\ast}}$ generated by Algorithm 3.1 converges strongly  to the
minimizer $q^{\ast}_{\infty}$ of problem
\eqref{constrained_min_med}-\eqref{vp_state_constraint_med}. For the purpose
we need some auxiliary results.

\begin{lem}\label{lem_strong_convergence_aux01}
Let $\{V_{k}\times
V_{k,\Gamma_{i}}\}$ be a sequence of discrete spaces generated by
Algorithm \ref{alg_afem_flux-recon}. If the sequence
$\{q_{k}\}\subset\bigcup_{k\geq 0}V_{k,\Gamma_{i}}$ weakly converges
to some $q^{\ast}\in Q_{\infty}$ in $L^{2}(\Gamma_{i})$, then there
exists a subsequence $\{q_{m}\}$ with $m=k_{n}$, such that for the
sequence $\{u_{m}(q_{m})\}\subset\bigcup_{k\geq 0}V_{k}$ produced by
\eqref{vp_state_constraint_disc} with $h$ replaced by $m$ and
$u{_\infty}(q^{\ast})\in V_{\infty}$ generated by
\eqref{vp_state_constraint_med} with $q=q^{\ast}$ there holds
\begin{equation}\label{strong-convergence_aux01}
    u_{m}(q_{m})\rightarrow u{_\infty}(q^{\ast})\quad\mbox{in
}L^{2}(\Gamma).
\end{equation}
\end{lem}

\begin{proof}
    Taking $\phi_{k}=u_{k}(q_{k})$ in \eqref{vp_state_constraint_disc}, we immediately
    know that $\|u_{k}(q_{k})\|_{1}$ is uniformly bounded independently
    of $k$ and hence there exist a subsequence, denoted by $\{u_{m}(q_{m})\}$ with $m=k_{n}$, and
    some $u^{\ast}\in H^{1}(\Omega)$ such that
    \begin{equation}\label{lem_strong_convergence_aux01_proof01}
        u_{m}(q_{m})\rightarrow u^{\ast}\quad\mbox{weakly in }H^{1}(\Omega),\qquad
        u_{m}(q_{m})\rightarrow u^{\ast}\quad\mbox{in }L^{2}(\Gamma).
    \end{equation}
    We only need to show $u^{\ast}=u_{\infty}(q^{\ast})$. As $V_{\infty}$ is weakly closed, $u^{\ast}\in V_{\infty}$. For any positive integer $l$,
    when we choose $m\geq l$, we know from \eqref{vp_state_constraint_disc} that
    \[
        (\alpha\boldsymbol{\nabla}u_{m}(q_{m}),\boldsymbol{\nabla}\phi_{l})+(\gamma u_{m}(q_{m}),\phi_{l})_{\Gamma_{a}}=(f,\phi_{l})+(\gamma u_{a},\phi_{l})_{\Gamma_{a}}-(q_{m},\phi_{l})_{\Gamma_{i}}\quad\forall~\phi_{l}\in
        V_{l}.
    \]
    Letting $m$ go to infinity and noting the convergence results in \eqref{lem_strong_convergence_aux01_proof01}
    as well as the weak convergence of $\{q_{k}\}$, we find
    \[
        (\alpha\boldsymbol{\nabla}u^{\ast},\boldsymbol{\nabla}\phi_{l})+(\gamma u^{\ast},\phi_{l})_{\Gamma_{a}}=(f,\phi_{l})+(\gamma u_{a},\phi_{l})_{\Gamma_{a}}-(q^{\ast},\phi_{l})_{\Gamma_{i}}\quad\forall~\phi_{l}\in
        V_{l}.
    \]
    As $l$ and $\phi_{l}\in V_{l}$ are arbitrary, we further obtain
    \[
        (\alpha\boldsymbol{\nabla}u^{\ast},\boldsymbol{\nabla}\phi)+(\gamma u^{\ast},\phi)_{\Gamma_{a}}=(f,\phi)+(\gamma u_{a},\phi)_{\Gamma_{a}}-(q^{\ast},\phi)_{\Gamma_{i}}\quad\forall~\phi\in
        V_{\infty},
    \]
    which leads to the desired conclusion.
\end{proof}

\begin{lem}\label{lem_strong_convergence_aux02}
Let $\{V_{k}\times V_{k,\Gamma_{i}}\}$ be a sequence of discrete spaces generated by
Algorithm \ref{alg_afem_flux-recon}. If the sequence
$\{q_{k}\}\subset\bigcup_{k\geq 0}V_{k,\Gamma_{i}}$ strongly
converges to some $q^{\ast}\in Q_{\infty}$ in $L^{2}(\Gamma_{i})$,
then for the sequence $\{u_{k}(q_{k})\}\subset\bigcup_{k\geq
0}V_{k}$ given by \eqref{vp_state_constraint_disc} with $h$ replaced
by $k$ and $u{_\infty}(q^{\ast})\in V_{\infty}$ given by
\eqref{vp_state_constraint_med} with $q=q^{\ast}$ there holds
\begin{equation}\label{strong-convergence_aux02}
    u_{k}(q_{k})\rightarrow u{_\infty}(q^{\ast})\quad\mbox{in }H^{1}(\Omega).
\end{equation}
\end{lem}

\begin{proof}
    We begin with an auxiliary discrete problem: Find $u_{k}(q^{\ast})\in V_{k}$ such that
    \begin{equation}\label{lem_strong_convergence_aux02_proof01}
        a(u_{k}(q^{\ast}),\phi)=(f,\phi)+(\gamma u_{a},\phi)_{\Gamma_{a}}-(q^{\ast},\phi)_{\Gamma_{i}}\quad\forall~\phi\in
        V_{k}.
    \end{equation}
    Subtracting \eqref{lem_strong_convergence_aux02_proof01} from \eqref{vp_state_constraint_disc} with
    $\phi=u_{k}(q_{k})-u_{k}(q^{\ast})$ and using the trace theorem as well as the norm equivalence we come to
    the estimate
    \[
        \|u_{k}(q^{\ast})-u_{k}(q_{k})\|_{1}\leq
        C\|q^{\ast}-q_{k}\|_{0,\Gamma_{i}}.
    \]
    On the other hand, we note that \eqref{lem_strong_convergence_aux02_proof01}
    is a finite element approximation of \eqref{vp_state_constraint_med} with
    $q=q^{\ast}\in Q_{\infty}$, so the Cea's lemma admits an optimal
    approximation property
    \[
        \|u_{\infty}(q^{\ast})-u_{k}(q^{\ast})\|_{1}\leq C\inf_{v\in
            V_{k}}\|u_{\infty}(q^{\ast})-v\|_{1}.
    \]
    Finally, the desired convergence \eqref{strong-convergence_aux02} is
    the consequence of the above two estimates and the density of $\bigcup_{k\geq 0}V_{k}$ in
    $V_{\infty}$.
\end{proof}

Now we are in a position to show the first main result of this
section.

\begin{thm}\label{thm_conv_constrained_min_med} Let $\{V_{k}\times
V_{k,\Gamma_{i}}\}$ be a sequence of discrete spaces generated by
Algorithm \ref{alg_afem_flux-recon} and $\{q_{k}^{\ast}\}$ be the
corresponding sequence of minimizers to the discrete problem
\eqref{constrained_min_disc}-\eqref{vp_state_constraint_disc}. Then
the whole sequence $\{q_{k}^{\ast}\}$ converges strongly in
$L^{2}(\Gamma_{i})$ to the unique minimizer $q_{\infty}^{\ast}$ of
the problem
\eqref{constrained_min_med}-\eqref{vp_state_constraint_med}.
\end{thm}

\begin{proof}
    The fact that $\|q_{k}^{\ast}\|_{0,\Gamma_{i}}$ is uniformly bounded implies there
    exist a subsequence, also denoted by $\{q_{k}^{\ast}\}$ and some $q^{\ast}\in Q_{\infty}$ such that
    \begin{equation}\label{thm_conv_constrained_min_med_proof01}
        q_{k}^{\ast}\rightarrow q^{\ast}\quad\mbox{weakly in
        }L^{2}(\Gamma_{i}).
    \end{equation}
    Then from Lemma \ref{lem_strong_convergence_aux01}, we know by
    extracting a subsequence with $m=k_{n}$ that
    \begin{equation}\label{thm_conv_constrained_min_med_proof02}
        u^{\ast}_{m}(q_{m}^{\ast})\rightarrow u_{\infty}(q^{\ast})\in
    V_{\infty}\quad\mbox{ in }L^{2}(\Gamma_{a}).
    \end{equation}
    On the other hand, for any $q\in Q_{\infty}$ there exists a
    sequence $\{q_{l}\}\subset\bigcup_{l\geq 0} V_{l,\Gamma_{i}}$ such that
    \begin{equation}\label{thm_conv_constrained_min_med_proof03}
        \lim_{l\rightarrow\infty}\|q_{l}-q\|_{0,\Gamma_{i}}=0,
    \end{equation}
    which, by Lemma \ref{lem_strong_convergence_aux02} and the trace theorem, implies
    \begin{equation}\label{thm_conv_constrained_min_med_proof04}
        \lim_{l\rightarrow\infty}\|u_{l}(q_{l})-z\|^{2}_{0,\Gamma_{a}}=\|u_{\infty}(q)-z\|^{2}_{0,\Gamma_{a}}.
    \end{equation}
    Choosing $k\geq l$ for sufficiently large $l$ and noting the whole sequence $\{q_{k}^{\ast}\}$ are minimizers of $\mathcal{J}(\cdot)$ over $\{V_{k,\Gamma_{i}}\}$, we can derive
    \begin{align*}
        \mathcal{J}(q_{k}^{\ast})\leq\mathcal{J}(q_{l})=\frac{1}{2}\|u_{l}(q_{l})-z\|_{0,\Gamma_{a}}^{2}+\frac{\beta}{2}\|q_{l}\|^{2}_{0,\Gamma_{i}}.
    \end{align*}
    Then a collection of \eqref{thm_conv_constrained_min_med_proof01}-\eqref{thm_conv_constrained_min_med_proof04} gives
    \begin{align*}
        \mathcal{J}(q^{\ast})&=\frac{1}{2}\|u_{\infty}(q^{\ast})-z\|_{0,\Gamma_{a}}^{2}+\frac{\beta}{2}\|q^{\ast}\|^{2}_{0,\Gamma_{i}}\\
                                      &\leq\lim_{m\rightarrow\infty}\frac{1}{2}\|u_{m}(q_{m}^{\ast})-z\|_{0,\Gamma_{a}}^{2}
                                       +\liminf_{m\rightarrow\infty}\frac{\beta}{2}\|q_{m}^{\ast}\|^{2}_{0,\Gamma_{i}}\\
                                      &\leq\liminf_{m\rightarrow\infty}\mathcal{J}(q_{m}^{\ast})\leq\limsup_{m\rightarrow\infty}\mathcal{J}(q_{m}^{\ast})
                                       \leq\limsup_{k\rightarrow\infty}\mathcal{J}(q_{k}^{\ast})\leq\limsup_{l\rightarrow\infty}\mathcal{J}(q_{l})
                                       =\mathcal{J}(q)\quad\forall~q\in Q_{\infty},
    \end{align*}
    which indicates that $q^{\ast}=q^{\ast}_{\infty}$ is the unique minimizer of the problem
    \eqref{constrained_min_med}-\eqref{vp_state_constraint_med}. Then the
    whole sequence $\{q_{k}^{\ast}\}$ converges weakly  to $q^{\ast}_{\infty}$.
    Moreover the choice $q=q^{\ast}$ in the above estimate yields
    equality $\displaystyle{\lim_{m\rightarrow\infty}}\mathcal{J}(q_{m}^{\ast})=\mathcal{J}(q^{\ast})=\inf\mathcal{J}(Q_{\infty})$
    and it follows that $\displaystyle{\lim_{k\rightarrow\infty}}\mathcal{J}(q_{k}^{\ast})=\inf\mathcal{J}(Q_{\infty})$ for the whole sequence $\{q_{k}^{\ast}\}$.
    Similarly, the strong convergence in \eqref{thm_conv_constrained_min_med_proof02} also holds true
    for the whole sequence $\{u^{\ast}_{k}(q_{k}^{\ast})\}$. These two facts
    guarantee that
    $
        \displaystyle{\lim_{k\rightarrow\infty}}\|q_{k}^{\ast}\|^{2}_{0,\Gamma_{i}}=\|q^{\ast}_{\infty}\|^{2}_{0,\Gamma_{i}},
    $
    which, along with the weak convergence, implies the strong convergence.
\end{proof}

Like the continuous case, after we introduce a Lagrangian multiplier
$p_{\infty}\in V_{\infty}$ to relax the constraint
\eqref{vp_state_constraint_med}, the minimization problem
\eqref{constrained_min_med} is recast as a saddle-point problem of
the following Lagrangian functional over $V_{\infty}\times
V_{\infty}\times Q_{\infty}$:
\[
    \mathcal{L}(u_{\infty},p_{\infty},q)=\frac{1}{2}\|u_{\infty}-z\|^{2}_{0,\Gamma_{a}}+\frac{\beta}{2}\|q\|^{2}_{0,\Gamma_{i}}
                                        -a(u_{\infty},p_{\infty})+(f,p_{\infty})+(\gamma
                                        u_{a},p_{\infty})_{\Gamma_{a}}-(q,p_{\infty})_{\Gamma_{i}}.
\]
The minimizer $q_{\infty}^{\ast}$ of \eqref{constrained_min_med} and
the related state $u_{\infty}^{\ast}$ are determined by the following system:
\begin{align}
    &a(u^{\ast}_{\infty},\phi)=(f,\phi)+(\gamma u_{a},\phi)_{\Gamma_{a}}-(q_{\infty}^{\ast},\phi)_{\Gamma_{i}}\quad\forall~\phi\in
    V_{\infty}\label{vp_state_med}\\
    &a(p_{\infty}^{\ast},v)=(u_{\infty}^{\ast}-z,v)_{\Gamma_{a}}\quad\forall~v\in V_{\infty},\label{vp_costate_med}\\
    & (\beta q_{\infty}^{\ast}-p_{\infty}^{\ast},w)_{\Gamma_{i}}=0\quad\forall~w\in Q_{\infty}.\label{gateaux_med}
\end{align}

Finally for the above system, we have the second main result of this
section.

\begin{thm}\label{thm_conv_vp_med} Let $\{V_{k}\times V_{k,\Gamma_{i}}\}$
be a sequence of discrete spaces generated by Algorithm
\ref{alg_afem_flux-recon}, then the sequence
$\{(u^{\ast}_{k},p^{\ast}_{k},q_{k}^{\ast})\}$ of discrete solutions
converges to
$(u_{\infty}^{\ast},p_{\infty}^{\ast},q_{\infty}^{\ast})$, the
solution of the problem \eqref{vp_state_med}-\eqref{gateaux_med},
in the following sense:
\begin{equation}\label{conv_vp_med}
    \|u_{k}^{\ast}-u_{\infty}^{\ast}\|_{1}\rightarrow 0,\quad
    \|p_{k}^{\ast}-p_{\infty}^{\ast}\|_{1}\rightarrow 0,\quad
    \|q_{k}^{\ast}-q_{\infty}^{\ast}\|_{0,\Gamma_{i}}\rightarrow
    0\quad\mbox{as}~k\rightarrow\infty.
\end{equation}
\end{thm}

\begin{proof}
    The third convergence follows directly from Theorem
    \ref{thm_conv_constrained_min_med}. Then by Lemma
    \ref{lem_strong_convergence_aux02} we obtain the first result. It
    remains to show the second one. We introduce an auxiliary
    problem:
    Find $\widetilde{p}_{k}\in V_{k}$ such that
    \begin{equation}\label{thm_conv_vp_med_proof01}
        (\alpha\boldsymbol{\nabla}\widetilde{p}_{k},\boldsymbol{\nabla}v)+(\gamma\widetilde{p}_{k},v)_{\Gamma_{a}}
        =(u^{\ast}_{\infty}-z,v)_{\Gamma_{a}}\quad\forall~v\in V_{k}.
    \end{equation}
    Combining \eqref{vp_costate_disc} and \eqref{thm_conv_vp_med_proof01} with
    $v=\widetilde{p}_{k}-p_{k}^{\ast}$ and using the trace theorem
    as well as the norm equivalence we obtain
    \begin{equation}\label{thm_conv_vp_med_proof02}
        \|\widetilde{p}_{k}-p_{k}^{\ast}\|_{1}\leq
        C\|u^{\ast}_{k}-u^{\ast}_{{\infty}}\|_{1}.
    \end{equation}
    On the other hand, it is not difficult to find that the problem
    \eqref{thm_conv_vp_med_proof01} is a discrete version of
    \eqref{vp_costate_med}. Hence the Cea's lemma gives
    \begin{equation}\label{thm_conv_vp_med_proof03}
        \|p_{\infty}^{\ast}-\widetilde{p}_{k}\|_{1}\leq C\inf_{v\in
        V_{k}}\|p_{\infty}^{\ast}-v\|_{1}.
    \end{equation}
    The desired result comes readily from
    \eqref{thm_conv_vp_med_proof02},
    \eqref{thm_conv_vp_med_proof03}, the first convergence in \eqref{conv_vp_med} and the construction of $V_{\infty}$.
\end{proof}

\begin{rem} As a matter of fact, the convergence results in Theorem
\ref{thm_conv_vp_med} have no connections with any particular
strategy adopted in the module \textsf{MARK} as in the case of
linear elliptic problems \cite{mor3} \cite{nsv} \cite{siebert}.
So other marking strategies work also here; see section 5.3 for
details.
\end{rem}

\section{Convergence}

In this section, we shall establish the convergence of Algorithm
\ref{alg_afem_flux-recon} in the following senses:
(1) the discrete solutions $\{(u_{k}^{\ast},p_{k}^{\ast},q_{k}^{\ast})\}$
converges strongly to the true solution of the problem
\eqref{vp_state}-\eqref{gateaux_cont};
(2) the error estimator $\eta_{k}$ converges to zero.
With some properties of
adaptively generated triangulations and the error estimator stated
in section 5.1, the proof of our main results is presented in section 5.2.
We will discuss the generalizations of the current arguments
to other marking strategies in section 5.3.

\subsection{Preliminaries}
We first introduce a convenient classification of all elements generated during
an adaptive algorithm. For each mesh $\mathcal{T}_{k}$, we define \cite{siebert}:
$$
    \mathcal{T}_{k}^{+}:=\bigcap_{l\geq k}\mathcal{T}_{l} \quad \mbox{and} \quad
    \mathcal{T}_{k}^{0}:=\mathcal{T}_{k}\setminus\mathcal{T}_{k}^{+}.
$$
So $\mathcal{T}_{k}^{+}$ consists of all
elements not refined after the $k$-th iteration and the sequence
$\{\mathcal{T}_{k}^{+}\}$ satisfies
$\mathcal{T}_{l}^{+}\subset\mathcal{T}_{k}^{+}$ for all $k> l$. On
the other hand, all elements in $\mathcal{T}_{k}^{0}$ are refined at
least once after the $k$-th iteration, that is to say for any
$T\in\mathcal{T}_{k}^{0}$, there exists $l\geq k$ such that
$T\in\mathcal{T}_{l}$ but
$T\in\hspace{-10pt}\slash~\mathcal{T}_{l+1}$. Correspondingly, the
domain $\Omega$ is split into two parts covered by
$\mathcal{T}_{k}^{+}$ and $\mathcal{T}_{k}^{0}$ respectively, i.e.
\[
    \bar{\Omega}=\Omega(\mathcal{T}_{k}^{+})\cup\Omega(\mathcal{T}_{k}^{0}):=\Omega_{k}^{+}\cup\Omega_{k}^{0}.
\]
We also define a mesh-size function
$h_{k}:\bar{\Omega}\rightarrow\mathbb{R}^{+}$ almost everywhere by
$h_{k}(x)=h_{T}$ for $x$ in the interior of an element
$T\in\mathcal{T}_{k}$ and $h_{k}(x)=h_{F}$ for $x$ in the relative
interior of a face $F\in\mathcal{F}_{k}$. It is clear that
the sequence $\{h_{k}\}$ given by Algorithm
\ref{alg_afem_flux-recon} strictly decreases on the region refined
by the newest vertex bisection. In fact, we have the following
observations (Corollary 3.3, \cite{siebert}).

\begin{lem}\label{lem_convergence_zero_mesh-size}
Let $\chi_{k}^{0}$ be the characteristic function of
$\Omega_{k}^{0}$, then the definition of $\mathcal{T}_{k}^{0}$ implies that
\begin{equation}\label{convergence_zero_mesh-size}
    \lim_{k\rightarrow\infty}\|h_{k}\chi^{0}_{k}\|_{\infty}
    =\lim_{k\rightarrow\infty}\|h_{k}\|_{\infty,\Omega_{k}^{0}}=0\,.
\end{equation}
\end{lem}

In the subsequent convergence analysis, the sum of
$\eta_{T,k,1}$ and $\eta_{T,k,2}$ over $\mathcal{T}_{k}$ will be split by
$\mathcal{T}_{k}^{0}$ and $\mathcal{T}_{k}^{+}$, and with the help of
Lemma \ref{convergence_zero_mesh-size} and the local approximation
properties of a classic nodal interpolation operator \cite{ciarlet}, we are able to
control the relevant residual in $\Omega_{k}^{0}$
(see the proof of Lemma \ref{lem_weak*_convergence_residual} below).
For the remaining part, we have to resort
to the marking strategy \eqref{maximum_marking}, which implies that
the maximal error indicator in
$\mathcal{T}_{k}\setminus\mathcal{M}_{k}$ is dominated by the
maximal error indicator in $\mathcal{M}_{k}$. Therefore it is
necessary to study only the convergence behavior of the latter.

\begin{lem}\label{lem_convergence_zero_marking} Let $\{\mathcal{T}_{k},
V_{k}\times
V_{k,\Gamma_{i}},(u^{\ast}_{k},p^{\ast}_{k},q^{\ast}_{k})\}$ be the
sequence of meshes, finite element spaces and discrete solutions
produced by Algorithm \ref{alg_afem_flux-recon} and $\mathcal{M}_{k}$ the set  of
marked elements  given by
\eqref{maximum_marking}. Then the following convergence holds
    \begin{equation}\label{convergence_zero_marking}
    \lim_{k\rightarrow\infty}\max_{T\in\mathcal{M}_{k}}\eta_{T,k}(u^{\ast}_{k},p^{\ast}_{k},q^{\ast}_{k},f,u_{a},z)=0.
    \end{equation}
\end{lem}

\begin{proof}
    Let $\widetilde{T}$ be the element where the error indicator
    attains the maximum among $\mathcal{M}_{k}$. As $\widetilde{T}\in\mathcal{M}_{k}\subset\mathcal{T}_{k}^{0}$,
    the local quasi-uniformity of $\mathcal{T}_{k}$ and Lemma
    \ref{lem_convergence_zero_mesh-size} tell that
    \begin{equation}\label{lem_convergence_zero_marking_proof01}
        |D_{\widetilde{T}}|\leq C|\widetilde{T}|\leq
        C\|h_{k}^{2}\|_{\infty,\Omega^{0}_{k}}\rightarrow
        0\quad\mbox{as~}k\rightarrow\infty.
    \end{equation}
    By means of the inverse estimate and the triangle inequality, we can
    estimate the error indicator
    $\eta_{\widetilde{T},k}:=(\eta^{2}_{\widetilde{T},k,1}+\eta^{2}_{\widetilde{T},k,2})^{1/2}$
    as follows
    \begin{eqnarray*}
        \eta^{2}_{\widetilde{T},k,1}(u^{\ast}_{k},q^{\ast}_{k},f,u_{a})
        &\leq&
        C(\|u^{\ast}_{k}\|^{2}_{1,D_{\widetilde{T}}}+h_{\widetilde{T}}\|q^{\ast}_{k}\|^{2}_{0,\partial\widetilde{T}\cap\Gamma_{i}}
        +h^{2}_{\widetilde{T}}\|f\|^{2}_{0,\widetilde{T}}+h_{\widetilde{T}}\|u_{a}\|^{2}_{0,\partial\widetilde{T}\cap\Gamma_{a}})\\
        &\leq&C(\|u^{\ast}_{k}-u^{\ast}_{\infty}\|^{2}_{1}+\|u^{\ast}_{\infty}\|^{2}_{1,D_{\widetilde{T}}}+\|q^{\ast}_{k}-q^{\ast}_{\infty}\|^{2}_{0,\Gamma_{i}}+h_{\widetilde{T}}\|q^{\ast}_{\infty}\|^{2}_{0,\partial\widetilde{T}\cap\Gamma_{i}}\\
        &&\qquad+h_{\widetilde{T}}^{2}\|f\|^{2}_{0,\widetilde{T}}+h_{\widetilde{T}}\|u_{a}\|^{2}_{0,\partial\widetilde{T}\cap\Gamma_{a}})\,, \\
        \eta^{2}_{\widetilde{T},k,2}(p^{\ast}_{k},u^{\ast}_{k},z)&\leq&
        C(\|p^{\ast}_{k}\|^{2}_{1,D_{\widetilde{T}}}+\|u^{\ast}_{k}\|^{2}_{1,D_{\widetilde{T}}}+h_{\widetilde{T}}\|z\|^{2}_{0,\partial\widetilde{T}\cap\Gamma_{a}})\\
        &\leq&
        C(\|p^{\ast}_{k}-p^{\ast}_{\infty}\|^{2}_{1}+\|p^{\ast}_{\infty}\|^{2}_{1,D_{\widetilde{T}}}+\|u^{\ast}_{k}-u^{\ast}_{\infty}\|^{2}_{1}+\|u^{\ast}_{\infty}\|^{2}_{1,D_{\widetilde{T}}}+h_{\widetilde{T}}\|z\|^{2}_{0,\partial\widetilde{T}\cap\Gamma_{a}}).
\end{eqnarray*}
    Now the result follows from \eqref{conv_vp_med}, \eqref{lem_convergence_zero_marking_proof01} and the absolute continuity of $\|\cdot\|_{1}$ and $\|\cdot\|_{0,\Gamma}$ with respect to the Lebesgue measure.
\end{proof}

\begin{rem}\label{rem_stability_error-indicator}
By inverse estimates we can deduce the following stability estimates
for any $T\in\mathcal{T}_{k}$:
\begin{eqnarray}
    \eta_{T,k,1}(u_{k}^{\ast},q_{k}^{\ast},f,u_{a})
    &\leq& C(\|u^{\ast}_{k}\|_{1,D_{T}}+\|q^{\ast}_{k}\|_{0,\partial T\cap\Gamma_{i}}+\|f\|_{0,T}+\|u_{a}\|_{0,\partial
    T\cap\Gamma_{a}}), \label{stability_error-indicator1}\\
    \eta_{T,k,2}(p_{k}^{\ast},u_{k}^{\ast},z)
    &\leq& C(\|p^{\ast}_{k}\|_{1,D_{T}}+\|u^{\ast}_{k}\|_{1,D_{T}}+\|z\|_{0,\partial
    T\cap\Gamma_{a}}). \label{stability_error-indicator2}
\end{eqnarray}
\end{rem}

\begin{rem}
    From the proof of Lemma \ref{lem_convergence_zero_marking}, we know that the maximum strategy \eqref{maximum_marking} in the module \textsf{MARK} is not utilized. Therefore this lemma is valid also for other markings.
\end{rem}

\subsection{Main results}\label{sec:main}

Now we turn our attention to the main results of this work. It is not difficult to know that once we can prove
the solution triplet $(u^{\ast}_{\infty}, p^{\ast}_{\infty}, q^{\ast}_{\infty})$  to the system \eqref{vp_state_med}-\eqref{gateaux_med}
is the exact solution triplet $(u^{\ast}, p^{\ast}, q^{\ast})$ to the system \eqref{vp_state}-\eqref{gateaux_cont}
in some appropriate
norm, then our expected first convergence result,
namely the sequence of discrete solutions $\{(u_{k}^{\ast},p_{k}^{\ast},q_{k}^{\ast})\}$ generated by Algorithm \ref{alg_afem_flux-recon}
converges strongly to the true solution of the problem
\eqref{vp_state}-\eqref{gateaux_cont}, will follow immediately from
Theorem \ref{thm_conv_vp_med}.
To do so, we shall first show the two residuals with respect to $u^{\ast}_{k}$ as well as $p^{\ast}_{k}$ have
weak vanishing limits for $u^{\ast}_{\infty}$ and $p^{\ast}_{\infty}$ (see Lemmas \ref{lem_weak*_convergence_residual} and \ref{lem_zero_residual_med}).
It is worth noting that compared with the case of the direct boundary value problems,
the inverse problem under consideration involves a major difficulty, i.e.,
$u^{\ast}$ and $u^{\ast}_{\infty}$ are determined by different fluxes $q^{\ast}$ and $q^{\ast}_{\infty}$ respectively.
To overcome the difficulty,
we define an auxiliary pair $(u(q_{\infty}^{\ast}),p(q_{\infty}^{\ast}))$ through
\eqref{vp_state}-\eqref{vp_costate} with $q^{\ast}$ replaced by $q^{\ast}_{\infty}$.
Then we will show that the pair $(u(q_{\infty}^{\ast}),p(q_{\infty}^{\ast}))$ is the same as
the limiting pair $(u_{\infty}^{\ast},p_{\infty}^{\ast})$.

As stated above, the first two residuals with
respect to $u^{\ast}_{k}(q_{k}^{\ast})$ and
$p^{\ast}_{k}(q_{k}^{\ast})$ are defined by
\[
    <\mathcal{R}(u_{k}^{\ast}),\phi>:=(f,\phi)+(\gamma
    u_{a},\phi)_{\Gamma_{a}}-(q_{k}^{\ast},\phi)_{\Gamma_{i}}-a(u_{k}^{\ast},\phi)\quad\forall~\phi\in
    H^{1}(\Omega),
\]
\[
    <\mathcal{R}(p_{k}^{\ast}),v>:=(u_{k}^{\ast}-z,v)_{\Gamma_{a}}-a(p_{k}^{\ast},v)\quad\forall~v\in
    H^{1}(\Omega).
\]

Since $\{q_{k}^{\ast}\}$ is a converging sequence of
minimizers by Theorem \ref{lem_strong_convergence_aux01},
it is uniformly bounded in $L^{2}(\Gamma_{i})$, so are
$\{u_{k}^{\ast}\}$ and $\{p_{k}^{\ast}\}$  in $H^1(\Omega)$
by means of \eqref{vp_state_disc} and
\eqref{vp_costate_disc}. Thus, we know
$\{\mathcal{R}(u_{k}^{\ast})\}$ and $\{\mathcal{R}(p_{k}^{\ast})\}$
are two sequences of uniformly bounded linear functionals in
$H^{1}(\Omega)'$, namely there exist two constants independent of
$k$ such that
\begin{equation}\label{residual_uni-bounded}
    \|\mathcal{R}(u_{k}^{\ast})\|_{H^{1}(\Omega)'}\leq C_{un1},\quad
    \|\mathcal{R}(p_{k}^{\ast})\|_{H^{1}(\Omega)'}\leq C_{un2}.
\end{equation}
In addition, we can easily observe from \eqref{vp_state_disc}
and \eqref{vp_costate_disc} that
\begin{equation}\label{residual_orthogonality}
    <\mathcal{R}(u_{k}^{\ast}),v>=0 \quad \m{and} \q <\mathcal{R}(p_{k}^{\ast}),v>=0\quad\forall~v\in
    V_{k}.
\end{equation}

Using these relations, we can establish the following weak convergence.

\begin{lem}\label{lem_weak*_convergence_residual} The sequence
$\{(u^{\ast}_{k},p^{\ast}_{k},q^{\ast}_{k})\}$ produced by Algorithm
\ref{alg_afem_flux-recon} satisfies
\begin{equation}\label{weak*_convergence_residual}
    \lim_{k\rightarrow\infty}<\mathcal{R}(u_{k}^{\ast}),\phi>=0,\quad
     \lim_{k\rightarrow\infty}<\mathcal{R}(p_{k}^{\ast}),\phi>=0\quad\forall~\phi\in
    H^{1}(\Omega)\,.
\end{equation}
\end{lem}

\begin{proof}
We prove only the first result by borrowing some techniques from \cite{siebert},
as the second convergence can be done in the same manner. We easily see that
$\mathcal{T}_{l}^{+}\subset\mathcal{T}_{k}^{+}\subset\mathcal{T}_{k}$
for $k>l$. This implies
$\Omega_{l}^{0}=\Omega(\mathcal{T}_{k}\setminus\mathcal{T}_{l}^{+}):=\bigcup\{T\in\mathcal{T}_{k},T\in\hspace{-10pt}\slash~\mathcal{T}^{+}_{l}\}$
and any refinement of $\mathcal{T}_{k}$ does not affect any element
in $\mathcal{T}_{l}^{+}$. Now we set
$\Omega_{l}^{\ast}:=\bigcup\{T\in\mathcal{T}_{k},T\cap\Omega_{l}^{0}\neq\emptyset\}$
and
$\Omega_{k}^{\times}:=\bigcup\{T\in\mathcal{T}_{k},T\cap\Omega_{l}^{+}\neq\emptyset\}$,
and write $I_k$ and $I_{k}^{sz}$ for the Lagrange and Scott-Zhang
interpolations respectively associated with $V_{k}$ \cite{ciarlet} \cite{sz}.
Then for any $\psi\in C^{\infty}(\bar{\Omega})$, we can derive
for $w=\psi-I_{k}\psi\in H^{1}(\Omega)$
by using the orthogonality
\eqref{residual_orthogonality} and elementwise integration by parts that
\begin{align}\label{lem_weak*_convergence_residual_proof01}
    &\quad|<\mathcal{R}(u_{k}^{\ast},\psi)>|=|<\mathcal{R}(u^{\ast}_{k},\psi-I_{k}\psi)>|=|<\mathcal{R}(u^{\ast}_{k},w-I_{k}^{sz}w)>|\nonumber\\
    &\leq C\sum_{T\in\mathcal{T}_{k}}\eta_{T,k,1}(u^{\ast}_{k},q_{k}^{\ast},f,u_{a})\|\psi-I_{k}\psi\|_{1,D_{T}}\nonumber\\
    &=C\big(\sum_{T\in\mathcal{T}_{k}\setminus\mathcal{T}_{l}^{+}}\eta_{T,k,1}(u^{\ast}_{k},q_{k}^{\ast},f,u_{a})\|\psi-I_{k}\psi\|_{1,D_{T}}
    +\sum_{T\in\mathcal{T}_{l}^{+}}\eta_{T,k,1}(u^{\ast}_{k},q_{k}^{\ast},f,u_{a})\|\psi-I_{k}\psi\|_{1,D_{T}}\big)\,.
\end{align}
Using \eqref{stability_error-indicator1} and the uniform boundedness of
$\|u_{k}^{\ast}\|_{1}$ and $\|q_{k}^{\ast}\|_{0,\Gamma_{i}}$, we have
\begin{equation}\label{lem_weak*_convergence_residual_proof02}
    \big(\sum_{T\in\mathcal{T}_{k}\setminus\mathcal{T}_{l}^{+}}\eta^{2}_{T,k,1}(u^{\ast}_{k},q_{k}^{\ast},f,u_{a})\big)^{1/2}\leq
    C(\|u_{k}^{\ast}\|_{1}+\|q_{k}^{\ast}\|_{0,\Gamma_{i}}+\|f\|_{0}+\|u_{a}\|_{\Gamma_{a}})\leq \widetilde{C}
\end{equation}
where $\widetilde{C}$ is independent of $k$. Furthermore, we can apply
the local interpolation error estimate for $I_{k}$ \cite{ciarlet} and the monotonicity of
the mesh-size function $h_{k}$ to obtain
\begin{equation}\label{lem_weak*_convergence_residual_proof03}
    \|w-I_{k}\psi\|_{1,\Omega_{l}^{\ast}}\leq
    C\|h_{l}\|_{\infty,\Omega_{l}^{\ast}}\|\psi\|_{2},\quad
    \|\psi-I_{k}\psi\|_{1,\Omega_{k}^{\times}}\leq
    C\|h_{l}\|_{\infty,\Omega_{k}^{\times}}\|\psi\|_{2}\leq C\|\psi\|_{2}.
\end{equation}
Now it follows readily from \eqref{lem_weak*_convergence_residual_proof01}-\eqref{lem_weak*_convergence_residual_proof03} and the local
quasi-uniformity of $\mathcal{T}_{l}$ that
for any $k\geq l$,
\begin{equation}\label{lem_weak*_convergence_residual_proof04}
    |<\mathcal{R}(u_{k}),\psi>|\leq
    C_{1}\|\psi\|_{2}\|h_{l}\|_{\infty,\Omega_{l}^{0}}+C_{2}\|\psi\|_{2}\big(\sum_{T\in\mathcal{T}_{l}^{+}}\eta^{2}_{T,k,1}(u^{\ast}_{k},q_{k}^{\ast},f,u_{a})\big)^{1/2}.
\end{equation}
To proceed our estimation, we can choose for any given $\varepsilon>0$
some sufficiently large $l$ by using Lemma
\ref{lem_convergence_zero_mesh-size}  such that
\begin{equation}\label{lem_weak*_convergence_residual_proof05}
    \|h_{l}\|_{\infty,\Omega_{l}^{0}}\leq\frac{\varepsilon}{2C_{1}\|\psi\|_{2}}.
\end{equation}
In addition, the marking strategy \eqref{maximum_marking} and Lemma
\ref{lem_convergence_zero_marking} ensure that
\[
    \lim_{k\rightarrow\infty}\max_{T\in\mathcal{T}_{k}\setminus\mathcal{M}_{k}}\eta_{T,k}(u^{\ast}_{k},p^{\ast}_{k},q^{\ast}_{k},f,u_{a},z)\leq
    \lim_{k\rightarrow\infty}\max_{T\in\mathcal{M}_{k}}\eta_{T,k}(u^{\ast}_{k},p^{\ast}_{k},q^{\ast}_{k},f,u_{a},z)=0,
\]
which, together with
$\mathcal{T}_{l}^{+}\cap\mathcal{M}_{k}=\emptyset$, implies
$$
\displaystyle\lim_{k\rightarrow\infty}\max_{T\in\mathcal{T}_{l}^{+}}\eta_{T,k}(u^{\ast}_{k},p^{\ast}_{k},q^{\ast}_{k},f,u_{a},z)=0.
$$
Therefore, we can choose $K\geq l$ for some fixed $l$ such that when
$k\geq K$,
\begin{equation}\label{lem_weak*_convergence_residual_proof06}
    \max_{T\in\mathcal{T}_{l}^{+}}\eta_{T,k,1}(u^{\ast}_{k},q^{\ast}_{k},f,u_{a})\leq\max_{T\in\mathcal{T}_{l}^{+}}\eta_{T,k}(u^{\ast}_{k},p^{\ast}_{k},q^{\ast}_{k},f,u_{a},z)\leq
    \frac{\varepsilon}{2C_{2}\|\psi\|_{2}}|\mathcal{T}_{l}^{+}|^{-\frac 12}.
\end{equation}
Then we can see from
\eqref{lem_weak*_convergence_residual_proof04}-\eqref{lem_weak*_convergence_residual_proof06}
that $<\mathcal{R}(u_{k}),\psi>$ is controlled by $\varepsilon$ for any $k\geq K$ and
$\psi\in C^{\infty}(\bar{\Omega})$, i.e.,
\begin{equation}\label{lem_weak*_convergence_residual_proof07}
    \lim_{k\rightarrow\infty}<\mathcal{R}(u^{\ast}_{k}),\psi>=0\quad\forall~\psi\in
    C^{\infty}(\bar{\Omega}).
\end{equation}
This gives the first convergence in (\ref{weak*_convergence_residual}) by
the density of $C^{\infty}(\bar{\Omega})$ in $H^{1}(\Omega)$.
\end{proof}

\begin{rem}\label{rem_weak*_convergence_residual}
One may see from
the second estimate in \eqref{lem_weak*_convergence_residual_proof06}
that for a fixed $l$,
\begin{equation}\label{conv_estimator_complement}
    \lim_{k\rightarrow\infty}\eta_{k}(u^{\ast}_{k},p^{\ast}_{k},q^{\ast}_{k},f,u_{a},z,\mathcal{T}_{l}^{+})=0.
\end{equation}
This observation will be used in the subsequent proof of Theorem \ref{thm_convergence_estimator}.
\end{rem}

Lemma \ref{lem_weak*_convergence_residual}  yields a important direct consequence.
Indeed, we know from \eqref{conv_vp_med}
that for any $\phi\in H^{1}(\Omega)$ and $v\in H^{1}(\Omega)$,
\beqnx
    <\mathcal{R}(u_{\infty}^{\ast}),\phi>&:=&(f,\phi)+(\gamma
    u_{a},\phi)_{\Gamma_{a}}-(q_{\infty}^{\ast},\phi)_{\Gamma_{i}}-a(u_{\infty}^{\ast},\phi)=\lim_{k\rightarrow\infty}<\mathcal{R}(u_{k}^{\ast}),\phi>\,, \\
<\mathcal{R}(p_{\infty}^{\ast}),v>&:=&(u_{\infty}^{\ast}-z,v)_{\Gamma_{a}}-a(p_{\infty}^{\ast},v)=\lim_{k\rightarrow\infty}<\mathcal{R}(p_{k}^{\ast}),v>\,.
\eqnx
Then the application of Lemma \ref{lem_weak*_convergence_residual} leads
readily to the following results about the vanishing residuals associated with
$u_{\infty}^{\ast}(q_{\infty}^{\ast})$ and $p_{\infty}^{\ast}(q_{\infty}^{\ast})$.

\begin{lem}\label{lem_zero_residual_med} The solution of the problem
\eqref{vp_state_med}-\eqref{gateaux_med} satisfies
\begin{equation}\label{zero_residual_med}
    <\mathcal{R}(u_{\infty}^{\ast}),\phi>=0\quad \m{and} \q <\mathcal{R}(p_{\infty}^{\ast}),\phi>=0\quad\forall~\phi\in
    H^{1}(\Omega).
\end{equation}
\end{lem}

To continue our analysis, we now introduce two auxiliary continuous problems:

Find $u(q_{\infty}^{\ast})\in
H^{1}(\Omega)$ and $p(q^{\ast}_{\infty})\in H^{1}(\Omega)$ such that
\begin{eqnarray}
    a(u(q_{\infty}^{\ast}),\phi)&=&(f,\phi)+(\gamma
    u_{a},\phi)_{\Gamma_{a}}-(q_{\infty}^{\ast},\phi)_{\Gamma_{i}}\quad\forall~\phi\in H^{1}(\Omega)\,,
    \label{vp_state_aux}\\
    a(p(q^{\ast}_{\infty}),v)&=&(u(q_{\infty}^{\ast})-z,v)_{\Gamma_{a}}\quad\forall~v\in
H^{1}(\Omega). \label{vp_costate_aux}
\end{eqnarray}

\begin{lem}\label{lem_eqa_state&costate_aux} For
the solution $(u^{\ast}_{\infty},p^{\ast}_{\infty},q^{\ast}_{\infty})$
of the problem \eqref{vp_state_med}-\eqref{gateaux_med} and
the solutions $u(q_{\infty}^{\ast})$, $p(q^{\ast}_{\infty})$  of the
problems \eqref{vp_state_aux} and \eqref{vp_costate_aux}, there
hold that
\begin{equation}\label{eqa_state&costate_aux}
    u^{\ast}_{\infty}=u(q_{\infty}^{\ast}) \quad \m{and} \q
    p^{\ast}_{\infty}=p(q_{\infty}^{\ast})\quad\mbox{in
    } ~H^{1}(\Omega).
\end{equation}
\end{lem}

\begin{proof}
    The Poincar\'{e} inequality, \eqref{vp_state_aux} and Lemma \ref{lem_zero_residual_med} yield that
    \[
        C\|u(q_{\infty}^{\ast})-u^{\ast}_{\infty}\|_{1}\leq
        \sup_{\|\phi\|_{1}=1}a(u(q_{\infty}^{\ast})-u^{\ast}_{\infty},\phi)=
        \sup_{\|\phi\|_{1}=1}<\mathcal{R}(u_{\infty}^{\ast}),\phi>=0,
    \]
    so the first equality is proved.
    Then the second equality in (\ref{eqa_state&costate_aux}) follows from the first result,
    \eqref{vp_costate_aux} and the following estimates:
    \begin{align*}
        C\|p(q_{\infty}^{\ast})-p^{\ast}_{\infty}(q_{\infty}^{\ast})\|_{1}&\leq
        \sup_{\|v\|_{1}=1}a(p(q_{\infty}^{\ast})-p^{\ast}_{\infty},v)=\sup_{\|v\|_{1}=1}(u(q^{\ast}_{\infty})-z,v)_{\Gamma_{a}}-a(p^{\ast}_{\infty},v)\\
        &=\sup_{\|v\|_{1}=1}(u^{\ast}_{\infty}(q^{\ast}_{\infty})-z,v)_{\Gamma_{a}}-a(p^{\ast}_{\infty},v)=\sup_{\|v\|_{1}=1}<\mathcal{R}(p_{\infty}^{\ast}),v>=0.
    \end{align*}
\end{proof}

Now we are ready to present the first main result in this paper.
\begin{thm}\label{thm_conv_vp} Let $(u^{\ast},p^{\ast},q^{\ast})$ be the
solution of the problem \eqref{vp_state}-\eqref{gateaux_cont}. Then
Algorithm \ref{alg_afem_flux-recon} produces a sequence of discrete
solutions ${(u_{k}^{\ast},p_{k}^{\ast},q_{k}^{\ast})}$ which converges to
$(u^{\ast},p^{\ast},q^{\ast})$ in the following sense
\begin{equation}\label{conv_vp}
    \lim_{k\rightarrow\infty}\|u^{\ast}-u_{k}^{\ast}\|_{1}=0,\quad
    \lim_{k\rightarrow\infty}\|p^{\ast}-p_{k}^{\ast}\|_{1}=0,\quad
    \lim_{k\rightarrow\infty}\|q^{\ast}-q_{k}^{\ast}\|_{0,\Gamma_{i}}=0.
\end{equation}
\end{thm}
\begin{proof}
We first show
$q^{\ast}=q^{\ast}_{\infty}$, which, together with Theorem \ref{thm_conv_vp_med}, leads to the third convergence. By means of the definition of $Q_{\infty}$ in section\,\ref{sec:limit}, the trace theorem and the
density of $\bigcup_{k\geq 0}V_{k}$ in $V_{\infty}$,
it is not difficult to get
$p^{\ast}_{\infty}|_{\Gamma_{i}}\in Q_{\infty}$.
Then there exists a sequence $\{p_{k}\}\subset\bigcup_{k\geq 0}V_{k}$
such that $p_{k}\rightarrow p^{\ast}_{\infty}$ in $H^{1}(\Omega)$, which, together with the trace theorem, allows
\[
    p_{k}|_{\Gamma_{i}}\rightarrow p^{\ast}_{\infty}|_{\Gamma_{i}}\quad\mbox{in}~L^{2}(\Gamma_{i}).
\]
Thus we have from \eqref{gateaux_cont} and \eqref{gateaux_med} that
\begin{equation}\label{thm_conv_vp_proof01}
    \beta q^{\ast}=p^{\ast},\quad \beta
    q^{\ast}_{\infty}=p^{\ast}_{\infty}\quad\mbox{on } ~\Gamma_{i}.
\end{equation}
On the other hand, we deduce from \eqref{vp_state}-\eqref{vp_costate}
and \eqref{vp_state_aux}-\eqref{vp_costate_aux} that
\begin{equation}\label{thm_conv_vp_proof02}
    a(u(q_{\infty}^{\ast})-u^{\ast},\phi)=(q^{\ast}-q^{\ast}_{\infty},\phi)_{\Gamma_{i}}\quad\forall~\phi\in H^{1}(\Omega),
\end{equation}
\begin{equation}\label{thm_conv_vp_proof03}
    a(p(q_{\infty}^{\ast})-p^{\ast},v)=(u(q^{\ast}_{\infty})-u^{\ast},v)_{\Gamma_{a}}\quad\forall~v\in H^{1}(\Omega).
\end{equation}
By taking $\phi=p(q_{\infty}^{\ast})-p^{\ast}$  and $v=u(q^{\ast}_{\infty})-u^{\ast}$ respectively
in \eqref{thm_conv_vp_proof02} and \eqref{thm_conv_vp_proof03}, we derive
\[
    \|u(q^{\ast}_{\infty})-u^{\ast}\|^{2}_{0,\Gamma_{a}}=(q^{\ast}-q^{\ast}_{\infty},p(q_{\infty}^{\ast})-p^{\ast})_{\Gamma_{i}}.
\]
With \eqref{thm_conv_vp_proof01}, we are further led to
\begin{align}\label{thm_conv_vp_proof04}
    &\beta\|q^{\ast}-q^{\ast}_{\infty}\|^{2}_{0,\Gamma_{i}}+\|u(q^{\ast}_{\infty})-u^{\ast}\|^{2}_{0,\Gamma_{a}}
            =(q^{\ast}-q^{\ast}_{\infty},\beta q^{\ast}-\beta
    q^{\ast}_{\infty}+p(q_{\infty}^{\ast})-p^{\ast})_{\Gamma_{i}}\nonumber\\
    =&(q^{\ast}-q^{\ast}_{\infty},p(q_{\infty}^{\ast})-p^{\ast}_{\infty})_{\Gamma_{i}}\leq
\|q^{\ast}-q^{\ast}_{\infty}\|_{0,\Gamma_{i}}\|p(q_{\infty}^{\ast})-p^{\ast}_{\infty}\|_{0,\Gamma_{i}},
\end{align}
which, together with the second equality in \eqref{eqa_state&costate_aux}, implies
\[
    \|q^{\ast}-q^{\ast}_{\infty}\|_{0,\Gamma_{i}}\leq
    \beta^{-1}\|p(q_{\infty}^{\ast})-p^{\ast}_{\infty}\|_{0,\Gamma_{i}}\leq
C\beta^{-1}\|p(q_{\infty}^{\ast})-p^{\ast}_{\infty}\|_{1}=0.
\]
So the last convergence in \eqref{conv_vp} holds thanks to Theorem
\ref{thm_conv_vp_med}. Moreover, it follows directly from
\eqref{thm_conv_vp_proof02} that
\begin{equation}\label{thm_conv_vp_proof05}
u(q_{\infty}^{\ast})=u^{\ast}\quad\mbox{in }H^{1}(\Omega).
\end{equation}
Now the first convergence in \eqref{conv_vp_med} and the first equality in
\eqref{eqa_state&costate_aux} yield the first result in \eqref{conv_vp}, i.e.,
$u_{k}^{\ast}\to u^{\ast}_{\infty}=u^{\ast}$
in $H^{1}(\Omega)$ as $k\rightarrow\infty$. Similarly, we can show using \eqref{thm_conv_vp_proof03} and
\eqref{thm_conv_vp_proof05} that
$p(q_{\infty}^{\ast})=p^{\ast}$ in $H^{1}(\Omega)$, then the desired
second convergence in \eqref{eqa_state&costate_aux} follows from Theorem \ref{thm_conv_vp_med} and
Lemma \ref{lem_eqa_state&costate_aux}.
\end{proof}

With the help of Theorem \ref{thm_conv_vp} and the local efficiency
\eqref{local_efficiency}, we are ready to establish the second main result of this paper.
\begin{thm}\label{thm_convergence_estimator} The sequence
$\{\eta_{k}(u^{\ast}_{k},p^{\ast}_{k},q^{\ast}_{k},f,u_{a},z)\}$ of
the estimators generated by Algorithm \ref{alg_afem_flux-recon}
converges to zero.
\end{thm}

\begin{proof}
We split the estimator for $k\geq l$ as in the proof of Lemma
\ref{lem_weak*_convergence_residual} that
\begin{equation}\label{thm_convergence_estimator_proof01}
    \eta^{2}_{k}(u^{\ast}_{k},p^{\ast}_{k},q^{\ast}_{k},f,u_{a},z)=\eta_{k}^{2}(u^{\ast}_{k},p^{\ast}_{k},q^{\ast}_{k},f,u_{a},z,\mathcal{T}_{k}\setminus\mathcal{T}_{l}^{+})+\eta_{k}^{2}(u^{\ast}_{k},p^{\ast}_{k},q^{\ast}_{k},f,u_{a},z,\mathcal{T}_{l}^{+}).
\end{equation}
It follows from \eqref{vp_state_disc}-\eqref{vp_costate_disc} and the strong
convergence of $\{q_{k}^{\ast}\}$ that
$\|u_{k}^{\ast}\|_{1}$, $\|p_{k}^{\ast}\|_{1}$ and
$\|q_{k}^{\ast}\|_{0,\Gamma_{i}}$ are all uniformly bounded above by a constant $C_{stab}$.
Summing up the lower bound \eqref{local_efficiency} over all elements in $\mathcal{T}_{k}\setminus\mathcal{T}_{l}^{+}$, we obtain
\[
    \begin{split}
    \eta_{k}^{2}(u^{\ast}_{k},p^{\ast}_{k},q^{\ast}_{k},f,u_{a},z,\mathcal{T}_{k}\setminus\mathcal{T}_{l}^{+})
    &\leq C\sum_{T\in\mathcal{T}_{k}\setminus\mathcal{T}_{l}^{+}}\big(\|u^{\ast}-u^{\ast}_{k}\|^{2}_{0,\omega_{T}}+\|p^{\ast}-p^{\ast}_{k}\|^{2}_{0,\omega_{T}}+\|q^{\ast}-q^{\ast}_{k}\|^{2}_{0,\partial
    T\cap\Gamma_{i}}\\
    &\qquad\qquad\qquad+\mathrm{osc}^{2}_{k}(f,\omega_{T})+\mathrm{osc}^{2}_{k}(u^{\ast}_{k},q^{\ast}_{k},\partial T)+\mathrm{osc}^{2}_{h}(p^{\ast}_{k},u^{\ast}_{k},\partial
    T)\big)\\
    &\leq C\big(\|u^{\ast}-u^{\ast}_{k}\|_{1}^{2}+\|p^{\ast}-p^{\ast}_{h}\|^{2}_{0}+\|q^{\ast}-q^{\ast}_{h}\|^{2}_{0,\Gamma_{i}}\\
            &\q +\max_{T\in\mathcal{T}_{k}\setminus\mathcal{T}_{l}^{+}}h_{T}(\|f\|^{2}_{0}+\|u_{a}\|_{0,\Gamma_{i}}^{2}+\|z\|^{2}_{0,\Gamma_{a}}+\|u_{k}^{\ast}\|^{2}_{1}+\|p_{k}^{\ast}\|^{2}_{1}+\|q_{k}^{\ast}\|^{2}_{0,\Gamma_{i}})\big),\\
    &\leq C\big(\|u^{\ast}-u^{\ast}_{k}\|_{1}^{2}+\|p^{\ast}-p^{\ast}_{h}\|^{2}_{0}+\|q^{\ast}-q^{\ast}_{h}\|^{2}_{0,\Gamma_{i}}
            \\
            &\q +\max_{T\in\mathcal{T}_{k}\setminus\mathcal{T}_{l}^{+}}h_{T}(\|f\|^{2}_{0}
           +\|u_{a}\|_{0,\Gamma_{i}}^{2}+\|z\|^{2}_{0,\Gamma_{a}}+C^{2}_{stab})\big),
    \end{split}
\]
where we used the facts that $\bar{f}_{T}$,
$\bar{J}_{F1}(u^{\ast}_{k},q_{k}^{\ast})$ and
$\bar{J}_{F2}(p^{\ast}_{k},u^{\ast}_{k})$ are the best
$L^{2}$-projections onto constant spaces and $h_{F}\leq C h_{T}$ for
any $F\in\partial
    T\cap\mathcal{F}
_{h}(\Gamma)$.
To complete the proof, we recall that
$\max_{T\in\mathcal{T}_{k}\setminus\mathcal{T}_{l}^{+}}h_{T}\leq\|h_{l}\|_{\infty,\Omega_{l}^{0}}\rightarrow0$
as $l\rightarrow\infty$ by Lemma \ref{lem_convergence_zero_mesh-size} and the
monotonicity of $h_{k}$, and the convergences in (\ref{conv_vp})
and (\ref{conv_estimator_complement}),
hence we can require two terms in
\eqref{thm_convergence_estimator_proof01} to be smaller than any given
positive number once we fix a large $l$ and choose $k$ sufficiently large.
\end{proof}

\subsection{Generalizations to other marking strategies}
In this section we shall extend the convergence results of Algorithm \ref{alg_afem_flux-recon}
established in the previous section\,\ref{sec:main} to the cases when
the marking criterion \eqref{maximum_marking} in Algorithm \ref{alg_afem_flux-recon}
is replaced by three other popular marking strategies, i.e.,
the equidistribution strategy, the modified equidistribution strategy and the practical D\"{o}rfler strategy.

By carefully reviewing the previous
analysis, it is not difficult to discover that Theorems \ref{thm_conv_constrained_min_med}-\ref{thm_conv_vp_med} and
Lemmas \ref{lem_convergence_zero_mesh-size}-\ref{lem_convergence_zero_marking} are all independent of
any specific marking strategy, and the maximum strategy
\eqref{maximum_marking} is only used in the proof of Lemma
\ref{lem_weak*_convergence_residual} for the condition
\begin{equation}\label{marking_condition}
    \max_{T\in\mathcal{T}_{k}\setminus\mathcal{M}_{k}}\eta_{T,k}(u^{\ast}_{k},p^{\ast}_{k},q^{\ast}_{k},f,u_{a},z)\leq
    \max_{T\in\mathcal{M}_{k}}\eta_{T,k}(u^{\ast}_{k},p^{\ast}_{k},q^{\ast}_{k},f,u_{a},z)
\end{equation}
to hold. Therefore, it suffices for us to check
whether this condition \eqref{marking_condition} is satisfied also by the aforementioned three strategies.

\paragraph{The equidistribution strategy.} Given a parameter $\theta\in [0,1]$ and a tolerance \textsf{TOL},
this strategy selects a subset $\mathcal{M}_{k}$ of all such elements $T\in \mathcal{T}_{k}$ to mark, which satisfies
\begin{equation}\label{equidistribution_marking}
    \eta_{_{T,k}}(u^{\ast}_{k},p^{\ast}_{k},q^{\ast}_{k},f,u_{a},z)\geq\theta\,\textsf{TOL}/\sqrt{|\mathcal{T}_{k}|}.
 \end{equation}
In practice, if $\eta_{k}(u^{\ast}_{k},p^{\ast}_{k},q^{\ast}_{k},f,u_{a},z)\leq\textsf{TOL}$,
the adaptive algorithm is terminated. It is easy to verify that whenever
$\eta_{k}(u^{\ast}_{k},p^{\ast}_{k},q^{\ast}_{k},f,u_{a},z)>\textsf{TOL}$
the element with the maximal error indicator is always included in
$\mathcal{M}_{k}$ according to \eqref{equidistribution_marking}.
Hence, \eqref{marking_condition} holds for the equidistribution
strategy. Then arguing as in Theorem\,\ref{thm_convergence_estimator}  for
the case of the maximum strategy, we have the following similar conclusion.
\begin{thm}\label{thm_convergence_equidistribution-strategy} Let
$(u^{\ast},p^{\ast},q^{\ast})$ be the solution of the problem
\eqref{vp_state}-\eqref{gateaux_cont} and
$\{(u_{k}^{\ast},p_{k}^{\ast},q_{k}^{\ast})\}$ be a sequence of
discrete solutions produced by Algorithm \ref{alg_afem_flux-recon}
with \eqref{equidistribution_marking} in place of
\eqref{maximum_marking} in the module \emph{\textsf{MARK}}. Then for
a given tolerance \emph{\textsf{TOL}}, the following inequality holds
after a finite number of iterations:
\begin{equation}\label{convergence_equidistribution-strategy}
    \eta_{k}(u^{\ast}_{k},p^{\ast}_{k},q^{\ast}_{k},f,u_{a},z)\leq\emph{\textsf{TOL}}\,.
\end{equation}
\end{thm}

\paragraph{The modified equidistribution strategy.} Given a parameter $\theta\in [0,1]$,
this strategy selects a subset $\mathcal{M}_{k}$ of all such elements $T\in \mathcal{T}_{k}$ to mark, which satisfies
    \begin{equation}\label{modified-equidistribution_marking}
       \eta_{_{T,k}}(u^{\ast}_{k},p^{\ast}_{k},q^{\ast}_{k},f,u_{a},z)\geq\theta\eta_{k}(u^{\ast}_{k},p^{\ast}_{k},q^{\ast}_{k},f,u_{a},z)
       /\sqrt{|\mathcal{T}_{k}|}.
    \end{equation}
With this marking strategy, the convergence results \eqref{conv_vp} and Theorem \ref{thm_convergence_estimator}
still hold true for Algorithm \ref{alg_afem_flux-recon}
since we may easily observe that the modified equidistribution strategy satisfies \eqref{marking_condition}.

\paragraph{The practical D\"{o}rfler strategy.} Given a parameter $\theta\in (0,1]$, this strategy
marks a subset $\mathcal{M}_{k}$ of elements in $\mathcal{T}_{k}$ that satisfy
\begin{align}
    \eta_{k}(u^{\ast}_{k},p^{\ast}_{k},q^{\ast}_{k},f,u_{a},z,\mathcal{M}_{k})&\geq\theta\eta_{k}(u^{\ast}_{k},p^{\ast}_{k},q^{\ast}_{k},f,u_{a},z),\label{Dorfler_marking}\\
    \min_{T\in\mathcal{M}_{k}}\eta_{T,k}(u^{\ast}_{k},p^{\ast}_{k},q^{\ast}_{k},f,u_{a},z)&\geq
     \max_{T\in\mathcal{T}_{k}\setminus\mathcal{M}_{k}}\eta_{T,k}(u^{\ast}_{k},p^{\ast}_{k},q^{\ast}_{k},f,u_{a},z).\label{Dorfler_marking_aux}
\end{align}
We can easily verify that \eqref{Dorfler_marking_aux} ensures the condition \eqref{marking_condition},
so the convergence results \eqref{conv_vp} and Theorem \ref{thm_convergence_estimator}
still follow.

\section*{Concluding remarks}
We have investigated a new adaptive finite element method for distributed flux reconstruction proposed recently
in \cite{lxz}. It has been demonstrated that as the algorithm proceeds
the adaptive sequence of the discrete triplets generated by the algorithm
converges to the true flux in $L^{2}$-norm, the true state and costate variables
in $H^{1}$-norm and the relevant sequence of estimators also has a vanishing limit. The latter guarantees that
the adaptive algorithm may stop within any given tolerance after a finite number of iterations.
For the sake of convenience, convergence results are established in the case of the maximum strategy in the module \textsf{MARK}
and then extended to other more practical marking strategies.

In the course of the convergence analysis, we have employed some techniques from nonlinear optimizations
to derive an important auxiliary result: the sequence of adaptive triplets generated by the algorithm
converges strongly to some limiting triplet. We believe there exist similar results for other inverse
problems in terms of output least-squares formulations with PDE constraints, so may follow the same line
to study their related AFEMs.

The convergence theory developed here may be extended to some nonlinear inverse problems
such as the reconstruction of the Robin coefficient
on an inaccessible part of the
boundary from some accessible boundary measurement data on the basis of an adaptive finite element
method.


\begin{thebibliography}{99}

\bibitem{ao}
M. Ainsworth and J. T. Oden, \textit{A Posteriori Error Estimation
in Finite Element Analysis}, Pure and Applied Mathematics,
Wiley-Interscience, New York, 2000.

\bibitem{ali}
O. M. Alifanov, \textit{Inverse Heat Transfer Problems}, Springer,
Berlin, 1994.

\bibitem{br}
I. Babu\v{s}ka and W. Rheinboldt, \textit{Error estimates for
adaptive finite element computations}, SIAM J. Numer. Anal., 15
(1978), 736-754.

\bibitem{bvogelius}
I. Babu\u{s}ka and M. Vogelius, \textit{Feedback and adaptive finite
element solution of one-dimensional boundary value problem}, Numer.
Math., 44 (1984), 75-102.

\bibitem{bangerth}
W. Bangerth and A. Joshi, \textit{Adaptive finite element methods
for the solution of inverse problems in optical tomography}, Inverse
Problems, 24 (2008), 1-22.

\bibitem{ban}
W. Bangerth and R. Rannacher, \textit{Adaptive Finite Element
Methods for Differential Equations}, Lectures in Mathematics,
ETH-Z\"{u}rich. Birkh\"{a}user, Basel, 2003.

\bibitem{bkr}
R. Becker, H. Kapp and R. Rannacher, \textit{Adaptive finite element
methods for optimal control of partial differential equations: Basic
concept}, SIAM J. Control Optim., 39 (2000), 113-132.

\bibitem{beckervexler}
R. Becker and B. Vexler, \textit{A posteriori error estimation for
finite element discretization of parameter identification problems},
Numer. Math., 96 (2004), 435-459.

\bibitem{beilinajohnson}
L. Beilina and C. Johnson \textit{A posteriori error estimation in
computational inverse scattering}, Math. Models Methods Appl. Sci.,
15 (2005), 23-35.

\bibitem{beilina2}
L. Beilina and M. V. Klibanov, \textit{A posteriori error estimates
for the adaptivity technique for the Tikhonov functional and global
convergence for a coefficient inverse problem}, Inverse Problems, 26
(2010), 045012 (27pp).

\bibitem{beilina3}
L. Beilina and M. V. Klibanov, \textit{Reconstruction of dielectrics
from experimental data via a hybrid globally convergent/adaptive
algorithm}, Inverse Problems, 26 (2010), 125009 (30).

\bibitem{beilina4}
L. Beilina,  M. V. Klibanov and M. Y. Kokurin, \textit{Adaptivity
with relaxation for ill-posed problems and global convergence for a
coefficient inverse problem}, J. Math. Sci. 167 (2010), 279-325.

\bibitem{bdd}
P. Binev, W. Dahmen and R. DeVore, \textit{Adaptive finite element
methods with convergence rates}, Numer. Math., 97 (2004), 219-268.

\bibitem{ckns}
J. M. Cascon, C. Kreuzer, R. H. Nochetto and K. G. Siebert,
\textit{Quasi-optimal convergence rate for an adaptive finite
element method}, SIAM J. Numer. Anal., 46 (2008), 2524-2550.

\bibitem{ciarlet}
P. G. Ciarlet, \textit{Finite element methods for elliptic
problems}, North-Holland, Amsterdam, 1978.

\bibitem{cdn}
A. Cohen, R. DeVore and R. H. Nochetto, \textit{Convergence rates
for AFEM with $H^{-1}$ data}, Found. Comput. Math., published
online: June 29, 2012, DOI: 10.1007/s10208-012-9120-1.

\bibitem{dk}
E. Divo and J. S. Kapat, \textit{Multi-dimensional heat flux
reconstruction using narrow-band thermochromic liquid crystal
thermography}, Inverse Problems in Science and Engineering, 9
(2001), 537-559.

\bibitem{dor}
W. D\"{o}rfler, \textit{A convergent adaptive algorithm for
Poisson's equation}, SIAM J. Numer. Anal., 33 (1996), 1106-1124.

\bibitem{fyl}
T. Feng, N. Yan and W. Liu, \textit{Adaptive finite element methods
for the identification of distributed parameters in elliptic
equation}, Adv. Comput. Math., 29 (2008), 27-53.

\bibitem{ghik}
A. Gaevskaya, R. H. W. Hoppe, Y. Iliash and M. Kieweg
\textit{Convergence analysis of an adaptive finite element method
for distributed control problems with control constraints}, Proc.
Conf. Optimal Control for PDEs, Oberwolfach, Germany (G. Leugering
et al.; eds.), Birkh\"{a}user, Basel, 2007.

\bibitem{gkv}
A. Griesbaum, B. Kaltenbacher and B. Vexler, \textit{Efficient
computation of the Tikhonov regularization parameter by
goal-oriented adaptive discretization}, Inverse Problems, 24 (2008),
025025 (20pp).

\bibitem{hh}
M. Hinterm\"{u}ller, R. H. W. Hoppe, \textit{Goal-oriented
adaptivity in pointwise state constrained optimal control of partial
differential equations}, SIAM J. Control Optim., 48 (2010),
5468-5487.

\bibitem{hhik}
M. Hinterm\"{u}ller, R. H. Hoppe, Y. Iliash and M. Kieweg,
\textit{An a posteriori error analysis of adaptive finite element
methods for distributed elliptic control problems with control
constraints}, ESAIM, Control Optim. Calc. Var., 14 (2008), 540-560.

\bibitem{lxz}
J. Li, J. Xie and J. Zou, \textit{An adaptive finite element
reconstruction of distributed fluxes}, Inverse Problems, 27 (2011),
075009 (25pp).

\bibitem{llmt}
R. Li, W. Liu, H. Ma and T. Tang, \textit{Adaptive finite element
approximation for distributed elliptic optimal control problems}
SIAM J. Control Optim., 41 (2002), 1321-1349.

\bibitem{liuyan}
W. Liu and N. Yan, \textit{A posteriori error estimates for
distributed convex optimal control problems}, Adv. Comput. Math., 15
(2001), 285-309.

\bibitem{lions}
J. L. Lions, \textit{Optimal Control of Systems Governed by Partial
Differential Equations}, Springer, Berlin, 1971.

\bibitem{koss}
I. Kossaczky. \textit{A recursive approach to local mesh refinement
in two and three dimensions}, J. Comp. Appl. Math., 55 (1995),
275-288.

\bibitem{maubach}
J. M. Maubach, \textit{Local bisection refinement for n-simplicial
grids generated by reflection}. SIAM J. Sci. Comput., 16 (1995),
210-227.

\bibitem{mitc}
W. F. Mitchell, \textit{A comparison of adaptive refinement
technieques for elliptic problems}, ACM Trans. Math. Software, 15
(1989), 326-347.

\bibitem{mor2}
P. Morin, R. H. Nochetto and K. G. Siebert, \textit{Convergence of
adaptive finite element methods}, SIAM Rev., 44 (2002), 631-658.

\bibitem{mor3}
P. Morin, K. G. Siebert and A. Veeser, \textit{A basic convergence
result for conforming adaptive finite elements}, Math. Models
Methods Appl. Sci., 18 (2008), 707-737.

\bibitem{nsv}
R. H. Nochetto, K. G. Siebert and A. Veeser, \textit{Theory of
adaptive finite element methods: an introduction}, Multiscale,
Nonlinear and Adaptive Approximation (R. A. DeVore and A. Kunoth,
Eds), Springer, New York, 2009, 409-542.

\bibitem{sz}
L. R. Scott and S. Zhang, \textit{Finite element interpolation of
nonsmooth functions satisfying boundary conditions}, Math. Comp., 54
(1990), 483-493.

\bibitem{siebert}
K. G. Siebert, \textit{A convergence proof for adaptive finite
elements without lower bounds}, IMA J. Numer. Anal., 31 (2011),
947-970.

\bibitem{stev}
R. Stevenson, \textit{Optimality of a standard adaptive finite
element method}, Found. Comput. Math., 7 (2007), 245-269.

\bibitem{stev1}
R. Stevenson, \textit{The completion of locally refined simplicial
partitions created by bisection}, Math. Comp., 77 (2008), 227-241.

\bibitem{traxler}
C. Traxler, \textit{An algorithm for adaptive mesh refinement in $n$
dimensions}, Computing, 59 (1997), 115¨C137.

\bibitem{xiezou}
J. Xie and J. Zou,\textit{Numerical reconstruction of heat fluxes},
SIAM J. Numer. Anal., 43 (2005), 1504-1535.

\bibitem{ver}
R. Verf\"{u}rth, \textit{A Review of A Posteriori Estimation and
Adaptive Mesh-Refinement Techniques}, Wiley-Teubner, Chichester, New
York, Stuttgart, 1996.

\bibitem{zk}
N. Zabaras and S. Kang, \textit{On the solution of an ill-posed
inverse design solidification problem using minimization techniques
in finite and infinite dimensional spaces}, Int. J. Numer. Methods
Eng., 36 (1993), 3973-3990.

\bibitem{zl}
N. Zabaras and J. Liu, \textit{An analysis of two-dimensional linear
inverse heat transfer problems using an integral method}, Numer.
Heat Transfer, 13 (1988), 527-533.


\end{thebibliography}
\end{document}